\documentclass[12pt]{article}

\usepackage{amsmath, amssymb, booktabs, color, algorithm, algorithmic}
\usepackage{amsthm}
\usepackage{bm}
\usepackage{caption}
\usepackage{subcaption}
\usepackage{here}
\usepackage{comment}
\usepackage[dvipdfmx]{graphicx}

\newtheorem{thm}{Theorem}[section]
\newtheorem{prop}{Proposition}[section]
\newtheorem{lem}{Lemma}[section]

\newtheorem{prob}{Problem}[section]
\newtheorem{algo}{Algorithm}[section]
\newtheorem{assum}{Assumption}[section]

\newcommand{\argmin}{\operatornamewithlimits{argmin}}

\numberwithin{equation}{section}

\begin{document}
\makeatletter

\begin{center}
\large{\bf Almost Sure Convergence of Random Projected Proximal and Subgradient Algorithms for Distributed Nonsmooth Convex Optimization}\\
\small{This work was supported by the Japan Society for the Promotion of Science through a Grant-in-Aid for
Scientific Research (C) (15K04763).}
\end{center}\vspace{3mm}

\begin{center}
\textsc{Hideaki Iiduka}\\
Department of Computer Science, 
Meiji University,
1-1-1 Higashimita, Tama-ku, Kawasaki-shi, Kanagawa 214-8571 Japan. 
(iiduka@cs.meiji.ac.jp)
\end{center}

\vspace{2mm}

\footnotesize{
\noindent\begin{minipage}{14cm}
{\bf Abstract:}
Two distributed algorithms are described that enable all users connected over a network to cooperatively solve the problem of minimizing the sum of all users' objective functions over the intersection of all users' constraint sets, where each user has its own private nonsmooth convex objective function and closed convex constraint set, which is the intersection of a number of simple, closed convex sets. 
One algorithm enables each user to adjust its estimate by using a proximity operator of its objective function and the metric projection onto one set randomly selected from the simple, closed convex sets. 
The other is a distributed random projection algorithm that determines each user's estimate by using a subgradient of its objective function instead of the proximity operator. 
Investigation of the two algorithms' convergence properties for a diminishing step-size rule revealed that, under certain assumptions, the sequences of all users generated by each of the two algorithms converge almost surely to the same solution. 
Moreover, convergence rate analysis of the two algorithms is provided, and desired choices of the step size sequences such that the two algorithms have fast convergence are discussed. 
Numerical comparisons for concrete nonsmooth convex optimization support the convergence analysis and demonstrate the effectiveness of the two algorithms.
\end{minipage}
 \\[5mm]

\noindent{\bf Keywords:} {almost sure convergence, distributed nonsmooth convex optimization, metric projection, proximity operator, random projection algorithm, subgradient}\\
\noindent{\bf Mathematics Subject Classification:} {90C15, 90C25, 90C30}

\hbox to14cm{\hrulefill}\par


\section{Introduction}\label{sec:1}
Future network models have attracted a great deal of attention. The concept of the network model is completely different from that of a conventional client-server network model. While a conventional client-server network model explicitly distinguishes hosts providing services (servers) from hosts receiving services (clients), the network model considered here does not assign fixed roles to hosts. Hosts composing the network, referred to here as users, can be both servers and clients. Hence, the network can function as an {\em autonomous}, {\em distributed}, and {\em cooperative} system. Although there are several forms of networks (e.g., hybrid peer-to-peer networks) in which some operations are intentionally centralized, here, we focus on networks that do not have any centralized operations. Therefore, we need to use distributed mechanisms that can work in cooperation with each user and neighboring users for controlling the network. 

{\em Distributed optimization} (see, e.g., \cite[Chapter 8]{bert}, \cite[Parts 1 and 2]{bert1997}, \cite[Part II]{casa2009}, \cite[PART II]{censor1998}, \cite{nedic}, \cite{sri} and references therein) plays a crucial role in making future networks \cite{cai2007,gold2005,ken2011,pop2012}, such as wireless, sensor, and peer-to-peer networks, stable and highly reliable. One approach \cite[Chapter 5]{ken2011}, \cite{mail}, \cite[Chapter 2]{sri} to reach this goal is to model the utility and strategy of each user as a concave utility function and convex constraint set and solve the problem of maximizing the overall utility of all users over the intersection of all users' constraint sets.
This paper focuses on a constrained convex minimization problem, where each user has its own private {\em nonsmooth convex} objective function (i.e., minus nonsmooth concave utility function) and closed convex constraint set that is the intersection of a number of simple, closed convex sets (e.g., the intersection of affine subspaces, half-spaces, and hyperslabs). The constrained nonsmooth convex optimization problem covers the important situations in which each user's objective function is differentiable with a non-Lipschitz continuous gradient or not differentiable (e.g., the $L^1$-norm) and includes, for instance, the problem of minimizing the total variation of a signal over a convex set, Tykhonov-like problems with $L^1$-norms \cite[I. Introduction]{comb2007}, the classifier ensemble problem with sparsity and diversity learning \cite[Subsection 2.2.3]{yin2014_2}, \cite[Subsection 3.2.4]{yin2014}, which is expressed as $L^1$-norm minimization, and the minimal antenna-subset selection problem \cite[Subsection 17.4]{yamada2011}. The main objective of the present paper is to solve the constrained nonsmooth convex optimization problem including many real-world problems by using distributed optimization techniques.

Two distributed optimization algorithms are presented for solving the constrained convex minimization problem described above. At each iteration of one algorithm, each user first calculates the weighed average of its estimate and the estimates received from its neighboring users and then updates its estimate by using the weighted average, the {\em proximity operator} of its own private nonsmooth convex function, and the metric projection onto one constraint set randomly selected from a number of simple, closed convex sets. The other algorithm is obtained by replacing the proximity operator in the first algorithm with a {\em subgradient} of each user's nonsmooth convex function.

The two algorithms are performed on the basis of a framework \cite[(2a)]{lee2013} on local user communications and random observations \cite[(2b)]{lee2013}, \cite[(2)]{nedic2011} of the local constraints, which ensures that each user can observe one simple, closed convex set onto which the metric projection can be efficiently calculated. Accordingly, the two algorithms can be applied to two complicated cases: (Case 1) Each user does not know the whole form of its private constraint set in advance and can observe only one simple, closed convex set at each instance; (Case 2) Each user knows the whole form of its private constraint set in advance, and the constraint set is the intersection of a huge number of simple, closed convex sets, which means that a metric projection onto the constraint set cannot be calculated easily. See \cite[Section I]{lee2013}, \cite[Section 1]{nedic2011}, and references therein for details on applications of the two cases, including collaborative filtering for recommender systems and text classification problems.

Related random projection algorithms have been proposed for convex optimization over the intersection of a number of closed convex sets. The most relevant to the work in this paper is the first distributed random projected gradient algorithm \cite{lee2013} that was proposed for solving a constrained smooth convex minimization problem when each user's objective function is convex with a Lipschitz continuous gradient. The centralized random projected gradient and subgradient algorithms \cite{nedic2011} were proposed for minimizing a single objective function over the intersection of an arbitrary family of convex inequalities. The incremental constraint projection-proximal algorithm \cite{wang2015} uses both random subgradient updates and random constraint updates. While there have been no reports on distributed random projection algorithms for nonsmooth convex optimization, thanks to the useful ideas in \cite{lee2013,wang2015}, we can devise a distributed random projected proximal algorithm (Algorithm \ref{algo:2}). Moreover, on the basis of \cite{lee2013,nedic2011}, we can devise a distributed random projected subgradient algorithm (Algorithm \ref{algo:1}) that is a generalization of the first distributed random projected gradient algorithm \cite[(2a), (2b)]{lee2013}. 

This paper makes two contributions that build on previously reported results. 
It presents two novel distributed random projection algorithms for constrained nonsmooth convex optimization that are based on each user's local communications. This means that they can be implemented independently of the network topology and that each user can calculate the weighted average of its estimate and the neighboring users' estimates. The algorithms proceed by performing a proximal or subgradient step for each user's objective function at the weighted average and projecting onto one simple, closed convex set that is randomly selected from each user's local constraint sets. Since the metric projection is a special case of nonexpansive mapping, the algorithms are connected to previous fixed point optimization algorithms \cite{iiduka_siopt2013,iiduka_mp2014,iiduka_mp2015,iiduka_hishinuma_siopt2014} for convex optimization over fixed point sets of nonexpansive mappings. The previous algorithms work for convex optimization over the intersection of convex constraint sets that are not always simple. However, since they work only when each user makes the best use of its own private information, they cannot be applied to the case where each user does not know the explicit form of its constraint set in advance. In contrast, the proposed algorithms work even when each user randomly sets one projection selected from many projections (see also (Case 1)). Therefore, the proposed algorithms have wider applications than previous fixed point optimization algorithms.
 
The other contribution is an analysis of the proposed proximal and subgradient algorithms. In contrast to the convergence analysis \cite{lee2013} of the first distributed random projected gradient algorithm, smooth convex analysis, which has tractable properties due to the use of Lipschitz continuous gradients, cannot be applied to the convergence analyses of the two proposed algorithms, which optimize nonsmooth convex functions. However, convergence analyses of the two algorithms can be performed by using useful properties \cite[Propositions 12.16, 12.27, and 16.14]{b-c} (Proposition \ref{prop:1}) of the proximity operators and the subgradients of nonsmooth convex objective functions. Thanks to the supermartingale convergence theorem \cite[Proposition 8.2.10]{bert} (Proposition \ref{prop:2}) and the portmanteau lemma \cite[Lemma 2.2]{vaa2007} (Proposition \ref{prop:3}), it is guaranteed that, under certain assumptions, the sequences of all users generated by each of the two algorithms converge almost surely to the same solution to the constrained nonsmooth convex optimization problem considered in this paper (Theorems \ref{thm:2} and \ref{thm:1}). Moreover, the rates of convergence of the two algorithms (Proposition \ref{prop:thm:2}, \eqref{rate}, Proposition \ref{prop:thm:1}, and \eqref{rate1}) are provided to illustrate the two algorithms' efficiency. The convergence rate analysis leads to the choices of the step size sequences such that the two algorithms have fast convergence. The numerical results for the two algorithms are also provided to support the convergence and convergence rate analyses.

This paper is organized as follows. Section \ref{sec:2} gives the mathematical preliminaries and states the main problem. Section \ref{sec:4} presents the proposed random projected proximal algorithm for solving the main problem and describes its convergence properties for a diminishing step size. Section \ref{sec:3} presents the proposed random projected subgradient algorithm for solving the main problem and describes its convergence properties for a diminishing step size. Section \ref{sec:5} considers concrete nonsmooth convex optimization and numerically compares the behaviors of the two algorithms. Section \ref{sec:6} concludes the paper with a brief summary and mentions future directions for improving the proposed algorithms.

\section{Preliminaries}\label{sec:2}
\subsection{Definitions and propositions}\label{subsec:2.1}
Let $\mathbb{R}^d$ be a $d$-dimensional Euclidean space with inner product $\langle \cdot,\cdot \rangle$ and its induced norm $\| \cdot \|$, and let $\mathbb{R}_+^d := \{(x_1,x_2,\ldots,x_d) \in \mathbb{R}^d \colon x_i \geq 0 \text{ } (i=1,2,\ldots,d) \}$. Let $[W]_{ij}$ denote the $(i,j)$th entry of a matrix $W$. Let $\mathrm{Pr}\{X\}$ and $\mathrm{E}[X]$ denote the probability and the expectation of a random variable $X$.

The metric projection onto a nonempty, closed convex set $C \subset \mathbb{R}^d$ is denoted by $P_C$,
and it is defined for all $x\in \mathbb{R}^d$ by $P_C(x)\in C$ and
\begin{align*}
\left\| x - P_C(x)  \right\| = \mathrm{d}\left(x,C\right) := \inf \left\{\left\| x-y \right\| \colon y\in C \right\}. 
\end{align*}
The mapping $P_C$ satisfies the firm nonexpansivity condition \cite[Proposition 4.8]{b-c}; i.e., 
$\| P_C(x) - P_C(y) \|^2 + \| (x-P_C(x)) - (y-P_C(y)) \|^2 \leq \|x-y\|^2$ $(x,y \in \mathbb{R}^d)$.
The subdifferential \cite[Definition 16.1]{b-c} of $f \colon \mathbb{R}^d \to \mathbb{R}$ is the set-valued operator 
defined for all $x\in \mathbb{R}^d$ by 
\begin{align*}
\partial f \left(x \right) := \left\{ u\in \mathbb{R}^d \colon f\left(y\right) \geq f\left(x\right) + \langle y-x,u\rangle
\text{ } \left(y\in \mathbb{R}^d \right)  \right\}.
\end{align*}
The proximity operator \cite[Definition 12.23]{b-c} of a convex function $f \colon \mathbb{R}^d \to \mathbb{R}$, denoted by $\mathrm{prox}_f$, maps every $x\in \mathbb{R}^d$ to the unique minimizer of 
$f(\cdot) + (1/2) \| x - \cdot \|^2$; i.e.,
\begin{align*}
\mathrm{prox}_f (x) \in  \argmin_{y\in \mathbb{R}^d} \left[ f(y) + \frac{1}{2} \left\|x-y\right\|^2  \right]
\text{ } \left(x\in \mathbb{R}^d \right).
\end{align*}
The uniqueness and existence of $\mathrm{prox}_f(x)$ are guaranteed for all $x\in \mathbb{R}^d$ \cite[Definition 12.23]{b-c}.

\begin{prop}\label{prop:1}
Let $f\colon \mathbb{R}^d \to \mathbb{R}$ be convex.
Then, the following hold:
\begin{enumerate}
\item[{\em (i)}]
$\partial f(x)$ is nonempty for all $x\in \mathbb{R}^d$.
\item[{\em(ii)}]
Let $x,p\in \mathbb{R}^d$. $p = \mathrm{prox}_f (x)$ if and only if $x-p \in \partial f (p)$
(i.e., $\langle y-p,x-p  \rangle + f(p) \leq f(y)$ for all $y\in \mathbb{R}^d$).
\item[{\em(iii)}]
$\mathrm{prox}_f$ is firmly nonexpansive. 
\item[{\em (iv)}]
Let $L > 0$.
Then, $f$ is Lipschitz continuous with a Lipschitz constant $L$ 
if and only if $\|u\| \leq L$ for all $x\in \mathbb{R}^d$ and for all $u\in \partial f(x)$. 
\end{enumerate}
\end{prop}

\begin{proof}
(i)--(iii)
Propositions 16.14(ii) and 12.26 in \cite{b-c} lead to (i) and (ii).
Moreover, Proposition 12.27 in \cite{b-c} implies (iii). 

(iv)
Let $x,y\in \mathbb{R}^d$ and choose $u_x \in \partial f(x)$ and $u_y \in \partial f(y)$.
The definition of $\partial f$ guarantees that 
$\langle x-y, u_x \rangle \geq f(x) - f(y) \geq \langle x-y,u_y \rangle$, which, together with the Cauchy-Schwarz inequality, 
implies that $\| u_x \| \|x-y\| \geq f(x) - f(y) \geq - \| u_y \| \|x-y\|$.
If $\|u\| \leq L$ for all $z\in\mathbb{R}^d$ and for all $u\in \partial f(z)$, 
$f$ is Lipschitz continuous. 

Suppose that $f$ is Lipschitz continuous with a constant $L$ and assume that there exist $x\in \mathbb{R}^d$ and $u\in\partial f(x)$ such that $\|u\| > L$. Define $y := x + u/\|u\|$.
The definition of $\partial f$ means that $f(y) \geq f(x) + \langle y-x,u \rangle = f(x) + \|u\| > f(x) + L$.
Hence, we have a contradiction:
$L < f(y) - f(x) \leq L \|x-y\| = L$. 
Accordingly, $\|u\| \leq L$ for all $x\in \mathbb{R}^d$ and for all $u\in \partial f(x)$.
\qed
\end{proof}

The following propositions are needed to prove the main theorems.

\begin{prop}\label{prop:2}{\em [The supermartingale convergence theorem \cite[Proposition 8.2.10]{bert}]}
Let $(Y_k)_{k\geq 0}$, $(Z_k)_{k\geq 0}$, and $(W_k)_{k\geq 0}$ be sequences of nonnegative random variables and 
let $\mathcal{F}_k$ ($k\geq 0$) denote the collection $Y_0, Y_1, \ldots, Y_k$, $Z_0, Z_1, \ldots, Z_k$, and $W_0, W_1, \dots, W_k$.
Suppose that $\sum_{k=0}^\infty W_k < \infty$ almost surely and that almost surely, for all $k\geq 0$, 
\begin{align*}
\mathrm{E} \left[ Y_{k+1} | \mathcal{F}_k  \right] \leq Y_k - Z_k + W_k.
\end{align*}
Then, $\sum_{k=0}^\infty Z_k < \infty$ almost surely and $(Y_k)_{k\geq 0}$ converges almost surely to a nonnegative random variable $Y$.
\end{prop}

\begin{prop}{\em [The portmanteau lemma \cite[Lemma 2.2]{vaa2007}]}\label{prop:3}
Let $(Y_k)_{k\geq 0}$ be a sequence of random variables that converges in law to a random variable $Y$.
Then, $\limsup_{k\to\infty} \mathrm{Pr} \{ Y_k \in F \} \leq \mathrm{Pr} \{ Y \in  F\}$ for every closed set $F$.
\end{prop}

A directed graph $G := (V,E)$ is a finite nonempty set $V$ of nodes (users) and a collection $E$ of ordered pairs of distinct nodes from $V$ \cite[p. 394]{b-g1987}.
A directed graph is said to be strongly connected if, for each pair of nodes $i$ and $l$, there exists a directed path from $i$ to $l$
\cite[p. 394]{b-g1987}.

\subsection{Main problem and assumptions} 
Let us consider a constrained nonsmooth convex optimization problem that is distributed over a network of $m$ users, indexed by 
\begin{align*}
V := \left\{1,2,\ldots,m \right\}.
\end{align*}
User $i$ $(i\in V)$ has its own private function $f_i \colon \mathbb{R}^d \to \mathbb{R}$ and constraint set $X_i \subset \mathbb{R}^d$.
On the basis of \cite[Section II]{lee2013}, let us define the constraint set $X := \bigcap_{i=1}^m X_i$
for the whole network. 
Suppose that $X$ is the intersection of $n$ closed convex sets. 
Let $I := \{ 1,2,\ldots,n \}$, and let $I_i$ $(i\in V)$ be the partition of $I$ such that $I = \bigcup_{i=1}^m I_i$ and $I_i \cap I_l = \emptyset$ for all $i,l\in V$ with $i\neq l$.
Then, we can define $X_i$ by the intersection of closed convex sets $X_i^j$ $(j\in I_i)$; i.e.,
\begin{align*}
X_i := \bigcap_{j\in I_i} X_i^j \text{ } \left(i\in V \right) \text{ and } 
X := \bigcap_{i\in V} X_i = \bigcap_{i=1}^m \bigcap_{j\in I_i} X_i^j.
\end{align*}
Throughout this paper, the following is assumed.

\begin{assum}\label{assum:1}
Suppose that 
\begin{enumerate}
\item[{\em(A1)}]
$X_i^j \subset \mathbb{R}^d$ ($i\in V, j\in I_i$) is a closed convex set onto which the metric projection $P_{X_i^j}$ can be efficiently computed and $X := \bigcap_{i\in V} X_i \neq \emptyset$;
\item[{\em(A2)}]
$f_i \colon \mathbb{R}^d \to \mathbb{R}$ ($i\in V$) is convex;
\item[{\em(A3)}]
For all $i\in V$, there exists $M_i \in \mathbb{R}$ such that $\sup \{ \| g_i \| \colon x\in X_i, g_i \in \partial f_i (x) \} \leq M_i$.
\end{enumerate}
\end{assum}
Assumption (A3) is satisfied if $f_i$ $(i\in V)$ is polyhedral on $X_i$ or $X_i$ $(i\in V)$ is bounded 
\cite[p. 471]{bert}.
Proposition \ref{prop:1}(iv) guarantees that, if $f_i$ $(i\in V)$ is Lipschitz continuous on $X_i$, Assumption (A3) holds.

The following problem is discussed in the paper.
\begin{prob}\label{prob:1}
Under Assumption \ref{assum:1}, 
\begin{align*}
\text{minimize } f \left(x \right) := \sum_{i\in V} f_i \left(x\right) \text{ subject to } x \in X := \bigcap_{i\in V} X_i, 
\end{align*}
where one assumes that Problem \ref{prob:1} has a solution.
\end{prob}

The solution set of Problem \ref{prob:1} is denoted by
\begin{align*}
X^\star := \left\{ x^\star \in X \colon f \left(x^\star \right) = f^\star := \inf_{x\in X} f\left(x\right)  \right\}.
\end{align*}
The condition $X^\star \neq \emptyset$ holds, for example, when one of $X_i^j$s is bounded \cite[Corollary 8.31, Proposition 11.14]{b-c}.

The main objective of this paper is to present distributed optimization algorithms that enable each user to solve Problem \ref{prob:1} without using other users' private information. To address this goal, we assume that each user and its neighboring users form a network in which each user can transmit its estimate to its neighboring users. The network topology at time $k$ is expressed as a directed graph $G(k) := (V, E(k))$, where $E(k) \subset V \times V$ and $(i,j) \in E(k)$ stands for a link such that user $i$ receives information from user $j$ at time $k$. Let $N_i(k) \subset V$ be the set of users that send information to user $i$; i.e., 
\begin{align*}
N_i\left(k\right) := \left\{ j\in V \colon \left(i,j \right) \in E \left(k\right) \right\}
\end{align*}
and $i\in N_i(k)$ $(i\in V, k \geq 0)$. To consider Problem \ref{prob:1}, we need the following assumption \cite[Assumption 4]{lee2013}, which leads to the result \cite[Theorem 4.2]{ram2012} (Lemma \ref{lem:4.3}) used to prove the main theorems.
\begin{assum}\label{assum:4}
There exists $Q \geq 1$ such that the graph $(V, \bigcap_{l=0}^{Q-1} E(k+l))$ is strongly connected for all $k \geq 0$.
\end{assum}

Assumptions (A1) and (A2) imply that, for all $k \geq 0$, user $i$ $(i\in V)$ can determine its estimate by using a subgradient or proximity operator of $f_i$ and the metric projection onto a certain constraint set selected from its own constraint sets $X_i^{j}$ $(j\in I_i)$. This paper assumes that user $i$ $(i\in V)$ forms the metric projection on the basis of random observations of the local constraints; i.e., user $i$ observes a local constraint set at time $k$,
\begin{align*}
X_i^{\Omega_i \left(k\right)}, \text{ where } \Omega_i \left(k\right) \in I_i \text{ is a random variable,}
\end{align*}
and thus uses $P_{X_i^{\Omega_i (k)}}$.
To perform our convergence analyses, we assume the following \cite[Assumptions 2 and 3]{lee2013}:
\begin{assum}\label{assum:2}
The random sequences $(\Omega_i(k))_{k\geq 0}$ ($i\in V$) are independent and identically distributed
and independent of the initial points $x_i(0)$ ($i\in V$) in the algorithms considered in the paper.
Moreover, $\mathrm{Pr}\{ \Omega_i(k)=j \} > 0$ holds for $i\in V$ and for $j\in I_i$.
\end{assum}

\begin{assum}\label{assum:3}
There exists $c > 0$ such that, for all $i\in V$, for all $x\in \mathbb{R}^d$, and for all $k \geq 0$,
\begin{align*}
\mathrm{d} \left(x,X\right)^2 \leq c \mathrm{E} \left[ \mathrm{d} \left(x,X_i^{\Omega_i\left(k\right)} \right)^2  \right].
\end{align*}
\end{assum}
Assumption \ref{assum:3} holds when $X_i^j$ $(i\in V, j\in I_i)$ is a linear inequality or equality constraint,
or when $X$ has an interior point (see \cite[p. 223]{lee2013} and references therein).

We also make the following assumption \cite[Assumption 5]{lee2013} to present our algorithms.

\begin{assum}\label{assum:5}
For $k\geq 0$, user $i$ ($i\in V$) has the weighted parameters $w_{ij}(k)$ ($j\in V$) satisfying the following:
\begin{enumerate}
\item[{\em (i)}]
$w_{ij}(k) := [W(k)]_{ij} \geq 0$ for all $j\in V$ and $w_{ij}(k) = 0$ when $j\notin N_i(k)$; 
\item[{\em (ii)}]
There exists $w \in (0,1)$ such that $w_{ij}(k) \geq w$ for all $j\in N_i(k)$.
\item[{\em (iii)}]
$\sum_{j\in V} [W(k)]_{ij} = 1$ for all $i\in V$ and 
$\sum_{i\in V} [W(k)]_{ij} = 1$ for all $j\in V$.
\end{enumerate}
Moreover, user $i$ ($i\in V$) also has the step size sequence $(\alpha_k)_{k\geq 0} \subset (0,\infty)$ satisfying 
\begin{align*}
\text{{\em (C1)} } \sum_{k=0}^\infty \alpha_k = \infty \text{ and {\em (C2) }} \sum_{k=0}^\infty \alpha_k^2 < \infty.
\end{align*}
\end{assum}

Let $x_i(k) \in \mathbb{R}^d$ be the estimate of user $i$ at time $k$ (see \eqref{xiprox} and \eqref{xi} in Algorithms \ref{algo:2} and \ref{algo:1} for details of the definition of $(x_i(k))_{k\geq 0}$ $(i\in V)$).
To analyze the proposed algorithms, we use the expectation taken with respect to the past history of the algorithms 
defined as follows \cite[p. 224]{lee2013}.
Let $\mathcal{F}_k$ be the $\sigma$-algebra generated by the entire history of the algorithms up to time $(k-1)$ inclusively;
i.e., for all $k \geq 1$,
\begin{align}\label{F}
\mathcal{F}_k := \left\{ x_i \left(0\right) \colon i\in V   \right\} \cup \left\{ \Omega_i\left(l\right) \colon l \in \left[0,k-1\right], i\in V \right\},
\end{align}
where $\mathcal{F}_0 := \{ x_i(0) \colon i\in V \}$.

\section{Distributed random projected proximal algorithm}\label{sec:4}
This section presents the following proximal algorithm with random projections for solving Problem \ref{prob:1}
under Assumptions \ref{assum:1}--\ref{assum:5}.

\begin{algo}\label{algo:2}
\text{ }

\begin{enumerate}
\item[Step 0.]
User $i$ ($i\in V$) sets $x_i(0) \in \mathbb{R}^d$ arbitrarily. 
\item[Step 1.]
User $i$ ($i\in V$) receives $x_j(k)$ from its neighboring users $j\in N_i(k)$ and computes the weighted average 
\begin{align}\label{viprox}
v_i\left(k\right) := \sum_{j\in N_i\left(k\right)} w_{ij}\left(k\right) x_j\left(k\right).
\end{align}
User $i$ updates its estimate $x_i(k+1)$ by 
\begin{align}\label{xiprox}
x_i\left( k+1\right) := P_{X_i^{\Omega_i\left(k\right)}} \left( \mathrm{prox}_{\alpha_k f_i} \left(v_i\left(k\right) \right) \right).
\end{align}
The algorithm sets $k:= k+1$ and returns to Step 1.
\end{enumerate}
\end{algo}
The definition \eqref{F} implies that, given $\mathcal{F}_k$ $(k\geq 0)$, the collection $x_i(0), x_i(1), \ldots, x_i(k)$ and $v_i(0), v_i(1), \ldots, v_i(k)$ generated by Algorithm \ref{algo:2} is fully determined. Algorithm \ref{algo:2} enables user $i$ $(i\in V)$ to determine $x_i(k+1)$ by using its proximity operator at the wighted average $v_i(k)$ of the received $x_j(k)$ $(j\in N_i(k))$ and the metric projection onto a constraint set $X_i^{\Omega_i(k)}$ randomly selected from $X_i^j$. 

The following is a convergence analysis of Algorithm \ref{algo:2}.

\begin{thm}\label{thm:2}
Under Assumptions \ref{assum:1}--\ref{assum:5}, the sequence $(x_i(k))_{k\geq 0}$ ($i\in V$) generated by Algorithm \ref{algo:2}
converges almost surely to a random point $x^\star \in X^\star$.
\end{thm}

\subsection{Proof of Theorem \ref{thm:2}}\label{subsec:4.1}
Let us first show the following lemma, which is needed to prove Theorem \ref{thm:2}.

\begin{lem}\label{lem:4.1}
Suppose that Assumption \ref{assum:1} holds. 
Then, for all $x\in X$, for all $i\in V$, for all $z\in X_i$, for all $k\geq 0$, and for all $\tau, \eta, \mu > 0$,
\begin{align*}
\left\| x_i \left(k+1\right) -x  \right\|^2 
&\leq 
\left\| v_i \left(k\right) -x  \right\|^2 -2 \alpha_k \left(f_i\left(z\right) -f_i\left(x\right) \right) 
+ N\left(\tau, \eta\right)\alpha_k^2\\
&\quad + \tau \left\| v_i\left(k\right) - z \right\|^2 + \left( \eta + \mu -2 \right) \left\| v_i\left(k\right) - \mathrm{prox}_{\alpha_k f_i} \left(v_i\left(k\right) \right)\right\|^2\\
&\quad - \left(1 - \frac{1}{\mu} \right) \left\| x_i\left(k+1\right) - v_i\left(k\right) \right\|^2,
\end{align*}
where $N(\tau,\eta)$ is a positive number depending on $\tau$ and $\eta$.
\end{lem} 

\begin{proof}
Choose $x\in X$ and $i\in V$ arbitrarily and fix an arbitrary $k\geq 0$.
The firm nonexpansivity of $P_{X_i^{\Omega_i(k)}}$ and $x = P_{X_i^{\Omega_i(k)}}(x)$ imply that
\begin{align*}
\left\| x_i \left(k+1\right) -x  \right\|^2 
&= 
\left\| P_{X_i^{\Omega_i\left(k\right)}} \left(\mathrm{prox}_{\alpha_k f_i} \left(v_i\left(k\right) \right) \right) 
  - P_{X_i^{\Omega_i\left(k\right)}}\left(x\right) \right\|^2\\
&\leq \left\| \mathrm{prox}_{\alpha_k f_i} \left(v_i\left(k\right) \right) -x  \right\|^2 
  - \left\| \mathrm{prox}_{\alpha_k f_i} \left(v_i\left(k\right) \right) - x_i \left(k+1\right)  \right\|^2.
\end{align*}
Since Proposition \ref{prop:1}(ii) implies that 
$v_i(k) - \mathrm{prox}_{\alpha_k f_i} (v_i(k))\in \partial (\alpha_k f_i) (\mathrm{prox}_{\alpha_k f_i} (v_i(k)))$, 
the definition of $\partial f_i$ $(i\in V)$ ensures that
\begin{align*}
\left\langle x - \mathrm{prox}_{\alpha_k f_i} \left(v_i\left(k\right)\right), 
v_i \left(k\right) - \mathrm{prox}_{\alpha_k f_i} \left(v_i\left(k\right)\right) \right\rangle 
\leq \alpha_k \left\{ f_i\left(x\right) - f_i\left(\mathrm{prox}_{\alpha_k f_i} \left(v_i\left(k\right)\right)\right) \right\},
\end{align*}
which, together with $\langle x,y\rangle = (1/2)(\|x\|^2 + \|y\|^2 - \|x-y\|^2)$ $(x,y\in \mathbb{R}^d)$, means that 
\begin{align*}
\left\| x - \mathrm{prox}_{\alpha_k f_i} \left(v_i\left(k\right)\right) \right\|^2 
&\leq 
\left\|  x- v_i \left(k\right) \right\|^2 
- \left\| v_i \left(k\right) - \mathrm{prox}_{\alpha_k f_i} \left(v_i\left(k\right)\right) \right\|^2\\
&\quad + 2 \alpha_k \left\{ f_i\left(x\right) - f_i\left(\mathrm{prox}_{\alpha_k f_i} \left(v_i\left(k\right)\right)\right) \right\}.
\end{align*}
Accordingly, 
\begin{align}
\left\| x_i \left(k+1\right) -x  \right\|^2 
&\leq 
\left\| v_i \left(k\right) -x\right\|^2 + 2 \alpha_k \left\{ f_i\left(x\right) - f_i\left(\mathrm{prox}_{\alpha_k f_i} \left(v_i\left(k\right)\right)\right) \right\}\label{ineq:4.1}\\
&\quad - \left\| v_i \left(k\right) - \mathrm{prox}_{\alpha_k f_i} \left(v_i\left(k\right)\right) \right\|^2
 - \left\| \mathrm{prox}_{\alpha_k f_i} \left(v_i\left(k\right) \right) - x_i \left(k+1\right)  \right\|^2\nonumber.
\end{align}
Choose $z\in X_i$ arbitrarily and set $g_i(z) \in \partial f_i(z)$. 
The definition of $\partial f_i$ ensures that
\begin{align}\label{ineq:4.2}
\begin{split}
&\quad 2 \alpha_k \left\{ f_i\left(x\right) - f_i\left(\mathrm{prox}_{\alpha_k f_i} \left(v_i\left(k\right)\right)\right) \right\}\\
&= 2 \alpha_k \left( f_i\left(x\right) - f_i \left(z\right) \right) + 2 \alpha_k \left\{ f_i \left(z\right) - f_i\left(\mathrm{prox}_{\alpha_k f_i} \left(v_i\left(k\right)\right)\right) \right\}\\
&\leq 2 \alpha_k \left( f_i \left(x \right) - f_i \left(z\right) \right)
+ 2 \alpha_k \left\langle z - \mathrm{prox}_{\alpha_k f_i} \left(v_i\left(k\right)\right), g_i\left(z\right) \right\rangle.
\end{split}
\end{align}
From the Cauchy-Schwarz inequality and $2|a||b| \leq \tau a^2 + (1/\tau)b^2$ $(a,b\in \mathbb{R},\tau > 0)$
\cite[Inequality (8)]{nedic2011}, for all $\tau, \eta > 0$,
\begin{align}\label{ineq:4.3}
\begin{split}
&\quad 2 \alpha_k \left\langle z - \mathrm{prox}_{\alpha_k f_i} \left(v_i\left(k\right)\right), g_i\left(z\right) \right\rangle\\
&= 2 \alpha_k \left\langle z - v_i\left(k\right), g_i\left(z\right) \right\rangle
    + 2 \alpha_k \left\langle v_i\left(k\right) - \mathrm{prox}_{\alpha_k f_i} \left(v_i\left(k\right)\right), g_i\left(z\right) \right\rangle\\
&\leq \tau \left\| z - v_i\left(k\right) \right\|^2 
  + \left( \frac{1}{\tau} + \frac{1}{\eta} \right) \left\|g_i\left(z\right) \right\|^2 \alpha_k^2
  + \eta \left\|v_i\left(k\right) - \mathrm{prox}_{\alpha_k f_i} \left(v_i\left(k\right)\right) \right\|^2.
\end{split}                
\end{align}
Moreover, from $\| x+y\|^2 = \|x\|^2 + \|y\|^2 + 2 \langle x,y \rangle$ $(x,y\in \mathbb{R}^d)$, 
\begin{align*}
\left\| \mathrm{prox}_{\alpha_k f_i} \left(v_i\left(k\right) \right) - x_i \left(k+1\right)  \right\|^2
&= 
\left\| \mathrm{prox}_{\alpha_k f_i} \left(v_i\left(k\right) \right) - v_i \left(k\right)\right\|^2 
   + \left\|v_i \left(k\right) - x_i \left(k+1\right) \right\|^2\\
&\quad + 2 \left\langle  \mathrm{prox}_{\alpha_k f_i} \left(v_i\left(k\right) \right) - v_i \left(k\right), 
        v_i \left(k\right) - x_i \left(k+1\right)  \right\rangle,  
\end{align*}
which, together with the Cauchy-Schwarz inequality and $2|a||b| \leq \tau a^2 + (1/\tau)b^2$ $(a,b\in \mathbb{R},\tau > 0)$,
implies that, for all $\mu > 0$,
\begin{align}\label{ineq:4.4}
\begin{split}
\left\| \mathrm{prox}_{\alpha_k f_i} \left(v_i\left(k\right) \right) - x_i \left(k+1\right)  \right\|^2
&\geq \left(1-\mu \right)\left\| \mathrm{prox}_{\alpha_k f_i} \left(v_i\left(k\right) \right) - v_i \left(k\right)\right\|^2\\
&\quad + \left(1 - \frac{1}{\mu} \right) \left\|v_i \left(k\right) - x_i \left(k+1\right) \right\|^2.
\end{split}
\end{align}
Hence, \eqref{ineq:4.1}, \eqref{ineq:4.2}, \eqref{ineq:4.3}, and \eqref{ineq:4.4} guarantee that, for all $\tau,\eta,\mu > 0$,
\begin{align*}
\left\| x_i \left(k+1\right) -x  \right\|^2 
&\leq 
\left\| v_i \left(k\right) -x\right\|^2 + 2 \alpha_k \left( f_i \left(x \right) - f_i \left(z\right) \right)
+ \left( \frac{1}{\tau} + \frac{1}{\eta} \right) \left\|g_i\left(z\right) \right\|^2 \alpha_k^2\\
&\quad + \tau \left\| z - v_i\left(k\right) \right\|^2 
  + \left(\eta + \mu -2 \right) \left\|v_i\left(k\right) - \mathrm{prox}_{\alpha_k f_i} \left(v_i\left(k\right)\right) \right\|^2\\
&\quad - \left(1 - \frac{1}{\mu} \right) \left\|v_i \left(k\right) - x_i \left(k+1\right) \right\|^2.
\end{align*}
Since (A3) ensures that $N (\tau,\eta) := (1/\tau +  1/\eta) \max_{i\in V} (\sup \{ \|g_i(z)\|^2  \colon z\in X_i\})
=  (1/\tau +  1/\eta) \max_{i\in V} M_i^2 < \infty$, the above inequality leads to Lemma \ref{lem:4.1}.
\end{proof}

\begin{lem}\label{lem:4.2}
Suppose that Assumptions \ref{assum:1}, \ref{assum:3}, and \ref{assum:5} hold. 
Then, 
$\sum_{k=0}^\infty \|v_i(k) - \mathrm{prox}_{\alpha_k f_i} (v_i(k))\|^2 < \infty$, and
$\sum_{k=0}^\infty \mathrm{d}(v_i(k), X)^2 < \infty$ almost surely for all $i\in V$.
\end{lem}

\begin{proof}
Let us take $x = z := P_X (v_i(k))$ $(i\in V, k\geq 0)$.
Then, Lemma \ref{lem:4.1} guarantees that, 
for all $i\in V$, for all $k\geq 0$, and for all $\tau,\eta,\mu > 0$,
\begin{align*}
\left\| x_i \left(k+1\right) - P_X \left(v_i\left(k\right)\right)  \right\|^2 
&\leq 
\left\| v_i \left(k\right) - P_X \left(v_i\left(k\right)\right)  \right\|^2 + \tau \left\| v_i\left(k\right) - P_X \left(v_i\left(k\right)\right) \right\|^2\\
&\quad + \left( \eta + \mu -2 \right) \left\| v_i\left(k\right) - \mathrm{prox}_{\alpha_k f_i} \left(v_i\left(k\right)\right) \right\|^2\\
&\quad - \left(1 - \frac{1}{\mu} \right) \left\|v_i \left(k\right) - x_i \left(k+1\right) \right\|^2 + N\left(\tau, \eta\right)\alpha_k^2,
\end{align*}
which, together with $\mathrm{d}(v_i(k),X) = \| v_i(k) - P_X (v_i(k)) \|$ and 
$\mathrm{d}(x_i(k+1), X) \leq \| x_i (k+1) - P_X (v_i(k))\|$ $(i\in V, k\geq 0)$, implies that, for all $i\in V$, for all $k\geq 0$, and for all $\tau,\eta,\mu > 0$,
\begin{align*}
\mathrm{d}\left(x_i\left(k+1\right), X\right)^2 
&\leq 
\mathrm{d}\left(v_i\left(k\right), X\right)^2 + \tau \mathrm{d}\left(v_i\left(k\right), X\right)^2\\
&\quad + \left( \eta + \mu -2 \right) \left\| v_i\left(k\right) - \mathrm{prox}_{\alpha_k f_i} \left(v_i\left(k\right)\right) \right\|^2
  + N\left(\tau, \eta\right)\alpha_k^2\\
&\quad - \left(1 - \frac{1}{\mu} \right) \left\|v_i \left(k\right) - x_i \left(k+1\right) \right\|^2.    
\end{align*}
Since $x_i(k+1) \in X_i^{\Omega_i(k)}$ $(i\in V,k\geq 0)$ means that $\mathrm{d}(v_i(k), X_i^{\Omega_i(k)}) \leq \| v_i(k) - x_i(k+1)\|$ $(i\in V,k\geq 0)$,
for all $i\in V$, for all $k\geq 0$, for all $\tau,\eta > 0$, and for all $\mu > 1$,
\begin{align*}
\mathrm{d}\left(x_i\left(k+1\right), X\right)^2 
&\leq 
\mathrm{d}\left(v_i\left(k\right), X\right)^2 + \tau \mathrm{d}\left(v_i\left(k\right), X\right)^2\\
&\quad + \left( \eta + \mu -2 \right) \left\| v_i\left(k\right) - \mathrm{prox}_{\alpha_k f_i} \left(v_i\left(k\right)\right) \right\|^2
  + N\left(\tau, \eta\right)\alpha_k^2\\
&\quad - \left(1 - \frac{1}{\mu} \right) \mathrm{d}\left(v_i\left(k\right), X_i^{\Omega_i(k)} \right)^2.
\end{align*}
By taking the expectation in this inequality conditioned on $\mathcal{F}_k$ defined in \eqref{F}, 
we have that, 
for all $i\in V$, for all $k\geq 0$, for all $\tau,\eta > 0$, and for all $\mu > 1$, almost surely
\begin{align*}
\mathrm{E} \left[ \mathrm{d}\left(x_i\left(k+1\right), X\right)^2 \Big| \mathcal{F}_k \right]
&\leq
\mathrm{d}\left(v_i\left(k\right), X\right)^2 + \tau \mathrm{d}\left(v_i\left(k\right), X\right)^2\\
&\quad + \left( \eta + \mu -2 \right) \left\| v_i\left(k\right) - \mathrm{prox}_{\alpha_k f_i} \left(v_i\left(k\right)\right) \right\|^2
  + N\left(\tau, \eta\right)\alpha_k^2\\
&\quad - \left(1 - \frac{1}{\mu} \right)
\mathrm{E} \left[ \mathrm{d}\left(v_i\left(k\right), X_i^{\Omega_i(k)} \right)^2 \Big| \mathcal{F}_k \right].
\end{align*}
Accordingly, Assumption \ref{assum:3} leads to the finding that, almost surely,
for all $i\in V$, for all $k\geq 0$, for all $\tau,\eta > 0$, and for all $\mu > 1$, 
\begin{align*}
\mathrm{E} \left[ \mathrm{d}\left(x_i\left(k+1\right), X\right)^2 \Big| \mathcal{F}_k \right]
&\leq
\mathrm{d}\left(v_i\left(k\right), X\right)^2+ \left( \eta + \mu -2 \right) \left\| v_i\left(k\right) - \mathrm{prox}_{\alpha_k f_i} \left(v_i\left(k\right)\right) \right\|^2\\
&\quad + \left\{\tau - \frac{1}{c}\left(1 - \frac{1}{\mu} \right) \right\} \mathrm{d}\left(v_i\left(k\right), X \right)^2
+ N\left(\tau, \eta\right)\alpha_k^2.
\end{align*}
Let us take $\tau := 1/(6c)$, $\eta := 1/3$, and $\mu := 3/2$.
From $\eta + \mu -2 = -1/6$, $\tau - (1/c)(1 - 1/\mu) = -1/(6c)$, and the convexity of $\mathrm{d}(\cdot,X)^2$, 
almost surely for all $i\in V$ and for all $k\geq 0$,
\begin{align*}
\mathrm{E} \left[ \mathrm{d}\left(x_i\left(k+1\right), X\right)^2 \Big| \mathcal{F}_k \right]
&\leq 
\sum_{j=1}^m \left[W(k)\right]_{ij} \mathrm{d}\left(x_j\left(k\right), X\right)^2 
   - \frac{1}{6}\left\| v_i\left(k\right) - \mathrm{prox}_{\alpha_k f_i} \left(v_i\left(k\right)\right) \right\|^2\\
&\quad - \frac{1}{6c} \mathrm{d}\left(v_i\left(k\right), X \right)^2 + N\left(\frac{1}{6c}, \frac{1}{3} \right)\alpha_k^2,
\end{align*}
where $N(1/(6c), 1/3) < \infty$ is guaranteed from Assumption (A3) (also see proof of Lemma \ref{lem:4.1}).
Hence, Assumption \ref{assum:5} ensures that, almost surely, for all $k\geq 0$, 
\begin{align}
\mathrm{E} \left[ \sum_{i=1}^m \mathrm{d}\left(x_i\left(k+1\right), X\right)^2 \Bigg| \mathcal{F}_k \right]
&\leq 
\sum_{j=1}^m  \mathrm{d}\left(x_j\left(k\right), X\right)^2 
 - \frac{1}{6} \sum_{i=1}^m \left\| v_i\left(k\right) - \mathrm{prox}_{\alpha_k f_i} \left(v_i\left(k\right)\right) \right\|^2\nonumber\\
&\quad - \frac{1}{6c} \sum_{i=1}^m \mathrm{d}\left(v_i\left(k\right), X \right)^2 + mN\left(\frac{1}{6c}, \frac{1}{3} \right)\alpha_k^2.\label{key1}
\end{align}
Proposition \ref{prop:2} and (C2) lead to 
$\sum_{k=0}^\infty \sum_{i=1}^m \| v_i(k) - \mathrm{prox}_{\alpha_k f_i}(v_i(k))\|^2 < \infty$ and 
$\sum_{k=0}^\infty \sum_{i=1}^m \mathrm{d}(v_i(k), X)^2 < \infty$ almost surely.
This completes the proof.
\end{proof}

The following lemma can be proven by using \cite[Theorem 4.2]{ram2012}.

\begin{lem}\label{lem:4.3}
Suppose that Assumptions \ref{assum:4} and \ref{assum:5} hold, and 
define $e_i(k) := x_i(k+1) - v_i(k)$ for all $i\in V$ and for all $k\geq 0$.
If $\sum_{k=0}^\infty \alpha_k \|e_i(k)\| < \infty$ almost surely for all $i\in V$,
$\sum_{k=0}^\infty \alpha_k \| x_i(k) - x_j(k) \| < \infty$ almost surely for all $i,j\in V$.
\end{lem}

\begin{proof}
The definition of $e_i(k)$ $(i\in V,k\geq 0)$ and \eqref{viprox} imply that 
$x_i(k+1) = v_i(k) + e_i(k) = \sum_{j=1}^m [W(k)]_{ij} x_j(k) + e_i(k)$ $(i\in V,k\geq 0)$.
Replacing $\theta_{i,k}$ in \cite[Theorem 4.2]{ram2012} by $x_i(k)$ guarantees that 
$\sum_{k=0}^\infty \alpha_k \| x_i(k) - x_j(k) \| < \infty$ almost surely for all $i,j\in V$.
\end{proof}

Lemmas \ref{lem:4.2} and \ref{lem:4.3} lead to the following.

\begin{lem}\label{lem:4.4}
Suppose that Assumptions \ref{assum:1}, \ref{assum:4}, and \ref{assum:5} hold and 
define $\bar{v}(k) := (1/m) \sum_{l=1}^m v_l(k)$ for all $k\geq 0$.
Then, $\sum_{k=0}^\infty \| e_i(k)\|^2 < \infty$ and $\sum_{k=0}^\infty \alpha_k \| v_i(k) - \bar{v}(k)\| < \infty$
almost surely for all $i\in V$.
\end{lem}

\begin{proof}
Put $z_i(k) := P_X (v_i(k))$ $(i\in V,k\geq 0)$.
From the triangle inequality, $z_i(k) \in X \subset X_i^{\Omega_i(k)}$ $(i\in V,k\geq 0)$,
and the nonexpansivity of $P_{X_i^{\Omega_i(k)}}$ $(i\in V,k\geq 0)$,
\begin{align*}
\left\| e_i\left(k\right) \right\| 
&\leq 
\left\| x_i\left(k+1\right) - z_i\left(k\right)\right\| + \left\| z_i\left(k\right) - v_i\left(k\right)\right\|\\
&= \left\| P_{X_i^{\Omega_i\left(k\right)}} \left(\mathrm{prox}_{\alpha_k f_i} \left(v_i\left(k\right) \right) \right) 
    - P_{X_i^{\Omega_i\left(k\right)}} \left( z_i\left(k\right) \right) \right\|
    + \left\| z_i\left(k\right) - v_i\left(k\right)\right\|\\
&\leq  \left\| \mathrm{prox}_{\alpha_k f_i} \left(v_i\left(k\right) \right) -  z_i\left(k\right) \right\| 
    + \left\| z_i\left(k\right) - v_i\left(k\right)\right\|\\
&\leq   \left\| \mathrm{prox}_{\alpha_k f_i} \left(v_i\left(k\right) \right) -  v_i\left(k\right) \right\| 
    + 2 \left\| z_i\left(k\right) - v_i\left(k\right)\right\|,  
\end{align*}
which, together with $(a+b)^2 \leq 2(a^2 + b^2)$ $(a,b\in\mathbb{R})$ and the definition of $z_i(k)$ $(i\in V, k\geq 0)$, 
implies that, for all $i\in V$ and for all $k\geq 0$,
\begin{align*}
\left\| e_i\left(k\right) \right\|^2
\leq 2 \left\| \mathrm{prox}_{\alpha_k f_i} \left(v_i\left(k\right) \right) -  v_i\left(k\right) \right\|^2 
  + 4 \mathrm{d}\left( v_i\left(k\right), X  \right)^2.
\end{align*}
Accordingly, Lemma \ref{lem:4.2} ensures that $\sum_{k=0}^\infty \| e_i(k)\|^2 < \infty$
almost surely for all $i\in V$.

Moreover, from $ab \leq (1/2)(a^2 + b^2)$ $(a,b\in\mathbb{R})$, we find that 
$\alpha_k \| e_i(k)\| \leq (1/2)(\alpha_k^2 + \|e_i(k)\|^2)$ $(i\in V,k\geq 0)$.
Hence, (C2) ensures that, for all $i\in V$, almost surely
$\sum_{k=0}^\infty \alpha_k \| e_i (k) \|
\leq (1/2) (\sum_{k=0}^\infty \alpha_k^2 + \sum_{k=0}^\infty \|e_i (k)\|^2 ) < \infty$,
which, together with Lemma \ref{lem:4.3}, implies that, for all $i,j\in V$ almost surely
\begin{align}\label{ineq:3.4.1}
\sum_{k=0}^\infty \alpha_k \left\| x_i\left(k\right) - x_j\left(k\right) \right\| < \infty.
\end{align}
The definitions of $\bar{v}(k)$ and $v_l(k)$ $(l\in V,k\geq 0)$ and Assumption \ref{assum:5}(iii) guarantee that, 
for all $k\geq 0$,
$\bar{v}(k) := (1/m) \sum_{l=1}^m v_l(k) = (1/m) \sum_{l=1}^m (\sum_{j=1}^m [W(k)]_{lj} x_j(k))
= (1/m) \sum_{j=1}^m \sum_{l=1}^m [W(k)]_{lj} x_j(k) = (1/m) \sum_{j=1}^m x_j(k)$.
Accordingly, Assumption \ref{assum:5}(iii) and the triangle inequality ensure that, for all $i\in V$ and for all $k\geq 0$,
\begin{align*}
\left\| v_i\left(k\right) - \bar{v}\left(k\right) \right\|
&\leq \sum_{j=1}^m \left[W\left(k\right)\right]_{ij} \left\| x_j\left(k\right) - \frac{1}{m} \sum_{l=1}^m x_l\left(k\right) \right\|,
\end{align*}
which, together with $[W(k)]_{ij} \leq 1$ $(i,j\in V)$, $x_j(k)= (1/m)\sum_{l=1}^m x_j(k)$ $(j\in V,k\geq 0)$, and the convexity of $\|\cdot\|$, implies that, 
for all $i\in V$ and for all $k\geq 0$,
\begin{align*}
\left\| v_i\left(k\right) - \bar{v}\left(k\right) \right\| 
\leq \frac{1}{m} \sum_{j=1}^m \sum_{l=1}^m \left\| x_j\left(k\right) - x_l\left(k\right) \right\|.
\end{align*}
Therefore, \eqref{ineq:3.4.1} leads to $\sum_{k=0}^\infty \alpha_k \| v_i(k) - \bar{v}(k)\| < \infty$
almost surely for all $i\in V$.
This completes the proof.
\end{proof}

Next, let us show the following lemma.

\begin{lem}\label{lem:4.5}
Suppose that the assumptions in Theorem \ref{thm:2} are satisfied, and define $z_i(k) := P_X(v_i(k))$ for all $i\in V$ and for all $k\geq 0$ and 
$\bar{z}(k) := (1/m) \sum_{i=1}^m z_i(k)$ for all $k\geq 0$.
Then, the sequence $(\| x_i(k) - x^\star \|)_{k\geq 0}$ converges almost surely for all $i\in V$ and for all $x^\star \in X^\star$
and $\liminf_{k\to\infty} f(\bar{z}(k)) = f^\star$ almost surely.
\end{lem}

\begin{proof}
Choose $x^\star \in X^\star$ arbitrarily.
The convexity of $\|\cdot\|^2$ and Assumption \ref{assum:5} imply that 
$\sum_{i=1}^m \| v_i(k) - x^\star\|^2 \leq \sum_{j=1}^m \| x_j(k) - x^\star\|^2$ $(k\geq 0)$.
Lemma \ref{lem:4.1} implies that, for all $k\geq 0$ and for all $\tau, \eta, \mu > 0$,
\begin{align*}
\sum_{i=1}^m \left\| x_i \left(k+1\right) -x^\star  \right\|^2 
&\leq 
\sum_{i=1}^m \left\| x_i \left(k\right) -x^\star  \right\|^2 -2 \alpha_k \sum_{i=1}^m \left(f\left(z_i\left(k\right)\right) -f_i\left(x^\star\right) \right)\\
&\quad + \tau \sum_{i=1}^m \left\| v_i\left(k\right) - z_i\left(k\right) \right\|^2
       - \left(1 - \frac{1}{\mu} \right) \sum_{i=1}^m \left\| x_i\left(k+1\right) - v_i\left(k\right) \right\|^2\\
&\quad + \left( \eta + \mu -2 \right) \sum_{i=1}^m \left\| v_i\left(k\right) - \mathrm{prox}_{\alpha_k f_i} \left(v_i\left(k\right) \right)\right\|^2
+ mN\left(\tau, \eta\right)\alpha_k^2.
\end{align*}
From $z_i(k) \in X$ $(i\in V, k\geq 0)$, the convexity of $X$ ensures that $\bar{z}(k) \in X \subset X_i$ $(i\in V)$.
Accordingly, (A3) means that $\|\bar{g}_i(k) \| \leq M_i$ for all $\bar{g}_i(k) \in \partial f_i(\bar{z}(k))$ $(i\in V,k\geq 0)$.
The definition of $\partial f_i$ $(i\in V)$ and the Cauchy-Schwarz inequality thus guarantee that, for all $i\in V$ and for all $k\geq 0$,
$f_i(z_i(k)) -f_i(x^\star)
= 
f_i(z_i(k)) - f_i(\bar{z}(k)) + f_i(\bar{z}(k)) -  f_i(x^\star)
\geq \langle z_i(k) - \bar{z}(k), \bar{g}_i(k) \rangle + f_i(\bar{z}(k)) -  f_i(x^\star)
\geq - \bar{M} \| z_i(k) - \bar{z}(k)\| + f_i(\bar{z}(k)) -  f_i(x^\star)$,
where $\bar{M} := \max_{i\in V} M_i < \infty$.
Moreover, the convexity of $\|\cdot\|$ and the nonexpansivity of $P_X$ imply that, for all $i\in V$ and for all $k\geq 0$,
$\| z_i(k) - \bar{z}(k)\|
= \| (1/m) \sum_{l=1}^m (P_X (v_i (k)) - z_l(k))\|
\leq 
(1/m) \sum_{l=1}^m \| P_X (v_i(k)) - P_X(v_l(k))\|
\leq 
(1/m) \sum_{l=1}^m \| v_i(k) - v_l(k)\|$,
which, together with the triangle inequality, implies that, for all $i\in V$ and for all $k\geq 0$,
$\| z_i(k) - \bar{z}(k)\|
\leq 
\| v_i(k) - \bar{v}(k) \| + (1/m) \sum_{l=1}^m \| v_l(k) - \bar{v}(k) \|$,
where $\bar{v}(k) := (1/m) \sum_{l=1}^m v_l(k)$ $(k\geq 0)$.
Accordingly, for all $i\in V$ and for all $k\geq 0$,
\begin{align}\label{fi}
\begin{split}
f_i\left(z_i\left(k\right) \right) -f_i\left(x^\star \right)
&\geq 
- \bar{M} \left\| v_i\left(k\right) - \bar{v}\left(k\right) \right\| - \frac{\bar{M}}{m} \sum_{l=1}^m \left\| v_l\left(k\right) - \bar{v}\left(k\right) \right\|\\
&\quad + f_i \left(\bar{z}\left(k\right)\right) -  f_i\left(x^\star\right).
\end{split}
\end{align}
Hence, 
the definitions of $f$ and $f^\star$ imply that, for all $k\geq 0$ and for all $\tau,\eta,\mu > 0$,
\begin{align*}
\sum_{i=1}^m \left\| x_i \left(k+1\right) -x^\star  \right\|^2 
&\leq 
\sum_{i=1}^m \left\| x_i \left(k\right) -x^\star  \right\|^2 
+ 2 \bar{M} \alpha_k \sum_{i=1}^m \left\| v_i\left(k\right) - \bar{v}\left(k\right) \right\|\\
&\quad + 2 \bar{M} \alpha_k \sum_{l=1}^m \left\| v_l\left(k\right) - \bar{v}\left(k\right) \right\|
-2 \alpha_k \left(f \left(\bar{z}\left(k\right) \right) -f^\star \right)\\
&\quad + \tau \sum_{i=1}^m \left\| v_i\left(k\right) - z_i\left(k\right) \right\|^2
- \left(1 - \frac{1}{\mu} \right) \sum_{i=1}^m \left\| x_i\left(k+1\right) - v_i\left(k\right) \right\|^2\\
&\quad + \left( \eta + \mu -2 \right) \sum_{i=1}^m \left\| v_i\left(k\right) - \mathrm{prox}_{\alpha_k f_i} \left(v_i\left(k\right) \right)\right\|^2 + m N\left(\tau, \eta\right)\alpha_k^2,
\end{align*}
which, together with $\mathrm{d}(v_i(k),X) = \| v_i(k) - z_i(k) \|$ and 
$\mathrm{d}(v_i(k), X_i^{\Omega_i(k)}) \leq \| v_i(k) - x_i(k+1)\|$ $(i\in V,k\geq 0)$, implies that,
for all $k\geq 0$, for all $\tau,\eta > 0$, and for all $\mu > 1$,
\begin{align*}
\sum_{i=1}^m \left\| x_i \left(k+1\right) -x^\star  \right\|^2
&\leq 
\sum_{i=1}^m \left\| x_i \left(k\right) -x^\star  \right\|^2 
+ 4 \bar{M} \alpha_k \sum_{i=1}^m \left\| v_i\left(k\right) - \bar{v}\left(k\right) \right\|\\
&\quad 
-2 \alpha_k \left(f \left(\bar{z}\left(k\right) \right) -f^\star \right) 
+ \tau \sum_{i=1}^m \mathrm{d} \left(v_i\left(k\right), X \right)^2\\ 
&\quad - \left( 1 - \frac{1}{\mu} \right) \sum_{i=1}^m \mathrm{d} \left( v_i\left(k\right), X_i^{\Omega_i\left(k\right)}\right)^2\\
&\quad + \left( \eta + \mu -2 \right) \sum_{i=1}^m \left\| v_i\left(k\right) - \mathrm{prox}_{\alpha_k f_i} \left(v_i\left(k\right) \right)\right\|^2 + m N\left(\tau, \eta\right)\alpha_k^2. 
\end{align*}
Accordingly, Assumption \ref{assum:3} guarantees that, almost surely
for all $k\geq 0$, for all $\tau,\eta > 0$, and for all $\mu > 1$,
\begin{align}
\mathrm{E} \left[\sum_{i=1}^m \left\| x_i \left(k+1\right) -x^\star  \right\|^2 \Bigg| \mathcal{F}_k \right]
&\leq 
\sum_{i=1}^m \left\| x_i \left(k\right) -x^\star  \right\|^2 
+ 4 \bar{M} \alpha_k \sum_{i=1}^m \left\| v_i\left(k\right) - \bar{v}\left(k\right) \right\|\nonumber\\
&\quad 
-2 \alpha_k \left(f \left(\bar{z}\left(k\right) \right) -f^\star \right)\nonumber\\
&\quad + \left\{ \tau - \frac{1}{c}\left(1 - \frac{1}{\mu} \right) \right\} \sum_{i=1}^m \mathrm{d} \left(v_i\left(k\right), X \right)^2\nonumber\\ 
&\quad + \left( \eta + \mu -2 \right) \sum_{i=1}^m \left\| v_i\left(k\right) - \mathrm{prox}_{\alpha_k f_i} \left(v_i\left(k\right) \right)\right\|^2\nonumber\\
&\quad + m N\left(\tau, \eta\right)\alpha_k^2.\label{key2} 
\end{align}
Taking $\tau := 1/(6c)$, $\eta := 1/3$, and $\mu := 3/2$ (also see proof of Lemma \ref{lem:4.2}) 
in the inequality above leads to the finding that, almost surely, for all $k\geq 0$,
\begin{align*}
\mathrm{E} \left[\sum_{i=1}^m \left\| x_i \left(k+1\right) -x^\star  \right\|^2 \Bigg| \mathcal{F}_k \right]
&\leq 
\sum_{i=1}^m \left\| x_i \left(k\right) -x^\star  \right\|^2
-2 \alpha_k \left(f \left(\bar{z}\left(k\right) \right) -f^\star \right)\\ 
&\quad + 4 \bar{M} \alpha_k \sum_{i=1}^m \left\| v_i\left(k\right) - \bar{v}\left(k\right) \right\|
+ m N\left(\frac{1}{6c}, \frac{1}{3}\right)\alpha_k^2.
\end{align*}
Therefore, since $\bar{z}(k) \in X$ implies $f(\bar{z}(k)) - f^\star \geq 0$ $(k\geq 0)$,
Proposition \ref{prop:2}, (C2), and Lemma \ref{lem:4.4} ensure that 
$(\| x_i(k) -x^\star \|)_{k\geq 0}$ converges almost surely for all $i\in V$.
Moreover,
\begin{align}\label{eq:3.5.1}
\sum_{k=0}^\infty \alpha_k \left(f \left(\bar{z}\left(k\right) \right) -f^\star \right) < \infty
\end{align}
is almost surely satisfied.
Now, under the assumption that almost surely $\liminf_{k\to\infty} f(\bar{z}(k)) -f^\star > 0$, 
$k_1 > 0$ and $\gamma > 0$ can be chosen such that 
$f(\bar{z}(k)) -f^\star \geq \gamma$ almost surely for all $k \geq k_1$. 
Accordingly, \eqref{eq:3.5.1} and (C1) mean that, almost surely
\begin{align*}
\infty = \gamma \sum_{k=k_1}^\infty \alpha_k \leq \sum_{k=k_1}^\infty \alpha_k \left(f \left(\bar{z}\left(k\right) \right) -f^\star \right) < \infty,
\end{align*}
which is a contradiction. Therefore, $\liminf_{k\to\infty} f(\bar{z}(k)) \leq f^\star$ almost surely.
From $f(\bar{z}(k)) - f^\star \geq 0$ $(k\geq 0)$, $\liminf_{k\to\infty} f(\bar{z}(k)) = f^\star$ almost surely.
This completes the proof.
\end{proof}

Now we are in position to prove Theorem \ref{thm:2}. 

\begin{proof}
Lemma \ref{lem:4.5} guarantees the almost sure convergence of $(x_i(k))_{k\geq 0}$ $(i\in V)$.
From \eqref{viprox}, $(v_i(k))_{k\geq 0}$ $(i\in V)$ also converges almost surely.
The definition of $\bar{v}(k)$ $(k\geq 0)$ implies the almost sure convergence of $(\bar{v}(k))_{k\geq 0}$.
Moreover, Lemma \ref{lem:4.2} implies that, for all $i\in V$, $\sum_{k=0}^\infty \mathrm{d}(v_i(k),X)^2 = \sum_{k=0}^\infty \| v_i(k) - z_i(k)\|^2 < \infty$
almost surely; i.e., $\lim_{k\to\infty}\| v_i(k) - z_i(k)\| =0$ almost surely.
Accordingly, $(z_i(k))_{k\geq 0}$ $(i\in V)$ converges almost surely.
This implies that there exists $x^*\in\mathbb{R}^d$ such that $(\bar{z}(k))_{k\geq 0}$ converges almost surely to $x^*$;
i.e., $(\bar{z}(k))_{k\geq 0}$ converges in law to $x^*$.
Hence, the closedness of $X$ and Proposition \ref{prop:3} guarantee that $\limsup_{k\to\infty} \mathrm{Pr}\{ \bar{z}(k) \in X \} \leq \mathrm{Pr} \{ x^* \in X \}$.
Since the definition of $\bar{z}(k)$ $(k\geq 0)$ implies that $\mathrm{Pr} \{ \bar{z}(k) \in X \}=1$ $(k\geq 0)$,
we find that $\mathrm{Pr} \{ x^* \in X \}=1$.
Moreover, Lemma \ref{lem:4.5} and the continuity of $f$ ensure that, almost surely
\begin{align*}
f \left(x^* \right) =  \lim_{k\to\infty} f \left( \bar{z}(k) \right) = \liminf_{k\to\infty} f \left( \bar{z}(k) \right) = f^\star;
\text{ i.e., } x^* \in X^\star.
\end{align*}
The definitions of $\bar{v}(k)$ and $\bar{z}(k)$ $(k\geq 0)$ mean that, for all $k\geq 0$,
$\| \bar{v}(k) - \bar{z}(k) \| \leq (1/m) \sum_{i=1}^m \| z_i(k) - v_i(k)\|$,
which, together with $\lim_{k\to\infty}\| v_i(k) - z_i(k)\| =0$ almost surely for all $i\in V$, implies that
$\lim_{k\to\infty}\| \bar{v}(k) - \bar{z}(k)\| =0$ almost surely.
Accordingly, the almost sure convergence of $(\bar{z}(k))_{k\geq 0}$ to $x^* \in X^\star$ guarantees that 
$(\bar{v}(k))_{k\geq 0}$ also converges almost surely to the same $x^*\in X^\star$.

Since Lemma \ref{lem:4.4} implies that $\sum_{k=0}^\infty \alpha_k \| v_i(k) -\bar{v}(k)\|< \infty$ almost surely for all $i\in V$,
a discussion similar to the one for obtaining $\liminf_{k\to\infty} f(\bar{z}(k)) \leq f^\star$ almost surely (see proof of Lemma \ref{lem:4.5}) and (C1) 
guarantee that $\liminf_{k\to\infty} \| v_i(k) -\bar{v}(k)\| = 0$ almost surely for all $i\in V$.
Moreover, the triangle inequality implies that $\| v_i(k) - x^* \| \leq \| v_i(k) -\bar{v}(k) \| + \| \bar{v}(k) - x^*\|$
$(i\in V, k\geq 0)$.
Hence, from $\lim_{k\to\infty} \|\bar{v}(k) - x^*\| = 0$ and $\liminf_{k\to\infty} \| v_i(k) -\bar{v}(k)\| = 0$ $(i\in V)$ almost surely,
we find that $\liminf_{k\to\infty} \| v_i(k) - x^*\| = 0$ almost surely for all $i\in V$.
Therefore, the almost sure convergence of $(v_i(k))_{k\geq 0}$ $(i\in V)$ leads to the finding that, for all $i\in V$,
\begin{align}\label{almost_vi}
\lim_{k\to\infty} \left\| v_i\left(k\right) - x^* \right\| = 0 \text{ almost surely}.
\end{align}
Since Lemma \ref{lem:4.4} ensures that, for all $i\in V$, $\lim_{k\to\infty} \|e_i(k)\|^2 = \lim_{k\to\infty} \| x_i(k+1) - v_i(k) \|^2 = 0$ almost surely, $(x_i(k))_{k\geq 0}$ $(i\in V)$ converges almost surely to $x^* \in X^\star$. This completes the proof.
\end{proof}

\subsection{Convergence rate analysis for Algorithm \ref{algo:2}}\label{subsec:4.2}
The discussion in subsection \ref{subsec:4.1} leads to the finding that the sequence of the feasibility error 
$(\mathrm{d}(x_i(k),X)^2)_{k\geq 0}$ and the sequence of the iteration error $(\|x_i(k) - x^\star\|^2)_{k\geq 0}$
are stochastically decreasing in the sense of the inequalities in the following proposition.

\begin{prop}\label{prop:thm:2}
Suppose that the assumptions in Theorem \ref{thm:2} hold, $x^\star \in X^\star$ is a solution to Problem \ref{prob:1}, 
and $(x_i(k))_{k\geq 0}$ ($i\in V$) is the sequence generated by Algorithm \ref{algo:2}. 
Then, there exist $\beta^{(j)} > 0$ ($j=1,2,3,4,5$) such that, almost surely for all $k\geq 0$,
\begin{align*}
&\mathrm{E} \left[ \sum_{i=1}^m \mathrm{d} \left( x_i\left(k+1\right), X  \right)^2 \Bigg| \mathcal{F}_k  \right]
\leq  \sum_{i=1}^m \mathrm{d} \left( x_i\left(k\right), X  \right)^2 
      - \sum_{j=1,2} \beta^{(j)} \gamma_k^{(j)} 
      + \beta^{(3)} \gamma_k^{(3)},\\
&\mathrm{E} \left[ \sum_{i=1}^m \left\| x_i\left(k+1\right) - x^\star  \right\|^2 \Bigg| \mathcal{F}_k  \right]
\leq  \sum_{i=1}^m \left\| x_i\left(k\right) - x^\star  \right\|^2 
      - \sum_{j=1,2,5} \beta^{(j)} \gamma_k^{(j)} 
      + \sum_{j=3,4} \beta^{(j)} \gamma_k^{(j)},
\end{align*}
where $\gamma_k^{(1)} := \sum_{i=1}^m \mathrm{d}(v_i(k),X)^2$,
$\gamma_k^{(2)} := \sum_{i=1}^m \| v_i(k) - \mathrm{prox}_{\alpha_k f_i} (v_i(k)) \|^2$,
$\gamma_k^{(3)} := \alpha_k^2$,
$\gamma_k^{(4)} := \alpha_k \sum_{i=1}^m \|v_i(k) - \bar{v}(k)  \|$, and
$\gamma_k^{(5)} := \alpha_k (f(\bar{z}(k)) -f^\star)$ ($k\geq 0$) satisfy $\sum_{k=0}^\infty \gamma_k^{(j)} < \infty$
($j=1,2,3,4,5$).
\end{prop}

\begin{proof}
Lemma \ref{lem:4.2} guarantees that 
$\gamma_k^{(1)} := \sum_{i=1}^m \mathrm{d}(v_i(k),X)^2$ and 
$\gamma_k^{(2)} := \sum_{i=1}^m \| v_i(k) - \mathrm{prox}_{\alpha_k f_i} (v_i(k)) \|^2$ $(k\geq 0)$ satisfy 
$\sum_{k=0}^\infty \gamma_k^{(j)} < \infty$ almost surely for $j=1,2$.
Condition (C2) implies that $\gamma_k^{(3)} := \alpha_k^2$ $(k\geq 0)$ satisfies $\sum_{k=0}^\infty \gamma_k^{(3)} < \infty$.
From Lemma \ref{lem:4.4} and \eqref{eq:3.5.1}, 
$\gamma_k^{(4)} := \alpha_k \sum_{i=1}^m \|v_i(k) - \bar{v}(k)  \|$ and
$\gamma_k^{(5)} := \alpha_k (f(\bar{z}(k)) -f^\star)$ $(k\geq 0)$ 
also satisfy 
$\sum_{k=0}^\infty \gamma_k^{(j)} < \infty$ almost surely for $j=4,5$.
Set $\tau := 1/(6c), \eta := 1/3$, and $\mu := 3/2$, and 
put $\beta^{(1)} := -(\tau - (1/c)(1-1/\mu)) =  1/(6c)$, 
$\beta^{(2)} := 2 - \eta - \mu = 1/6$, 
$\beta^{(3)} := m N(\tau, \eta)$,
$\beta^{(4)} := 4 \bar{M}$, 
and 
$\beta^{(5)} := 2$,
where $\beta^{(3)}, \beta^{(4)} < \infty$ hold from (A3) and $\bar{M} := \max_{i\in V} M_i$.
Accordingly, \eqref{key1} and \eqref{key2} ensure that Proposition \ref{prop:thm:2} holds.
\end{proof}

From $- \sum_{j=1,2} \beta^{(j)} \gamma_k^{(j)} + \beta^{(3)} \gamma_k^{(3)} \leq \beta^{(3)} \gamma_k^{(3)} = \beta^{(3)} \alpha_k^2$
and $- \sum_{j=1,2,5} \beta^{(j)} \gamma_k^{(j)} + \sum_{j=3,4} \beta^{(j)} \gamma_k^{(j)} \leq \sum_{j=3,4} \beta^{(j)} \gamma_k^{(j)}
= (\beta^{(3)} \alpha_k + \beta^{(4)} \sum_{i=1}^m \|v_i(k) - \bar{v}(k)  \|) \alpha_k$ $(k\geq 0)$,
Proposition \ref{prop:thm:2} and Theorem \ref{thm:2} (see also \eqref{almost_vi} for the almost sure boundedness of $(v_i(k))_{k\geq 0}$ $(i\in V)$) indicate that $(x_i(k))_{k\geq 0}$ $(i\in V)$ in Algorithm \ref{algo:2} converges almost surely to a solution to Problem \ref{prob:1} under the following convergence rates:
almost surely, for all $k\geq 0$,
\begin{align}\label{rate}
\begin{split}
&\mathrm{E} \left[ \sum_{i=1}^m \mathrm{d} \left( x_i\left(k+1\right), X  \right)^2 \Bigg| \mathcal{F}_k  \right]
\leq \sum_{i=1}^m \mathrm{d} \left( x_i\left(k\right), X  \right)^2 
+ O\left(\alpha_k^2\right),\\
&\mathrm{E} \left[ \sum_{i=1}^m \left\| x_i\left(k+1\right) - x^\star  \right\|^2 \Bigg| \mathcal{F}_k  \right]
\leq  \sum_{i=1}^m \left\| x_i\left(k\right) - x^\star  \right\|^2 
+ O \left(\alpha_k\right).
\end{split}
\end{align} 

From Proposition \ref{prop:thm:2}, the convergence rate of Algorithm \ref{algo:2} depends on $\beta^{(j)}$ and $(\gamma_k^{(j)})_{k\geq 0}$ $(j=1,2,3,4,5)$; i.e., the number of users $m$ and the step size sequence $(\alpha_k)_{k\geq 0}$. When $m$ is fixed, it is desirable to set $(\alpha_k)_{k\geq 0}$ so that, for all $k\geq 0$, $\mathrm{E}[ \sum_{i=1}^m \mathrm{d} ( x_i(k+1), X)^2 | \mathcal{F}_k] < \sum_{i=1}^m \mathrm{d}( x_i(k), X)^2$ and $\mathrm{E} [\sum_{i=1}^m \| x_i(k+1) - x^\star \|^2 | \mathcal{F}_k] < \sum_{i=1}^m \| x_i(k) - x^\star \|^2$ are almost surely satisfied. Accordingly, Algorithm \ref{algo:2} has fast convergence if $(\alpha_k)_{k\geq 0}$ can be chosen so as to satisfy
\begin{align*}
- \sum_{j=1,2} \beta^{(j)} \gamma_k^{(j)} + \beta^{(3)} \gamma_k^{(3)} < 0 \text{ and }
- \sum_{j=1,2,5} \beta^{(j)} \gamma_k^{(j)} + \sum_{j=3,4} \beta^{(j)} \gamma_k^{(j)} < 0 \text{ } (k\geq 0).
\end{align*}
Hence, it would be desirable to set $(\alpha_k)_{k\geq 0}$ so as to satisfy
$\sum_{j=3,4} \beta^{(j)} \gamma_k^{(j)} = \beta^{(3)} \alpha_k^2 + \beta^{(4)} \alpha_k \sum_{i=1}^m \|v_i(k) - \bar{v}(k)  \| \approx 0$ as much as possible, e.g., to set $\alpha_k := \alpha/(k+1)$ with a small positive constant $\alpha$.
Section \ref{sec:5} gives numerical examples such that Algorithm \ref{algo:2} with $\alpha_k := 10^{-3}/(k+1)$ has faster convergence than
Algorithm \ref{algo:2} with $\alpha_k := 1/(k+1)$.

\section{Distributed random projected subgradient algorithm}\label{sec:3}
This section presents the following subgradient algorithm with random projections for solving Problem \ref{prob:1}.

\begin{algo}\label{algo:1}
\text{ }

\begin{enumerate}
\item[Step 0.]
User $i$ ($i\in V$) sets $x_i(0) \in \mathbb{R}^d$ arbitrarily. 
\item[Step 1.]
User $i$ ($i\in V$) receives $x_j(k)$ from its neighboring users $j\in N_i(k)$ 
and computes the weighted average $v_i(k)$ defined as in \eqref{viprox}
and the subgradient $g_i(k) \in \partial f_i ( v_i(k))$. 
User $i$ updates its estimate $x_i(k+1)$ by 
\begin{align}\label{xi}
x_i\left( k+1\right) := P_{X_i^{\Omega_i\left(k\right)}} \left(v_i\left(k\right) - \alpha_k g_i\left(k\right)  \right).
\end{align}
The algorithm sets $k:= k+1$ and returns to Step 1.
\end{enumerate}
\end{algo}

In this section, Assumption (A3) is replaced with 
\begin{enumerate}
\item[(A3)']
For all $i\in V$, there exists $M_i \in \mathbb{R}$ such that $\sup \{ \| g_i \| \colon x\in X_i, g_i \in \partial f_i (x)\} \leq M_i$.
For all $i\in V$, there exists $C_i \in \mathbb{R}$ such that 
$\sup \{ \| g_i(k) \| \colon g_i(k) \in \partial f_i (v_i(k)), k \geq 0 \} \leq C_i$.
\end{enumerate}
It is obvious that Assumption (A3)' is stronger than Assumption (A3). Assumption (A3)' holds when $f_i$ $(i\in V)$ is Lipschitz continuous on $\mathbb{R}^d$ (Proposition \ref{prop:1}(iv)). The boundedness of $(g_i(k))_{k\geq 0}$ $(i\in V)$ is needed to show that Algorithm \ref{algo:1} satisfies $\sum_{k=0}^\infty \mathrm{d}(v_i(k),X)^2 < \infty$ almost surely for all $i\in V$, which is the essential part of the convergence analysis of Algorithm \ref{algo:1} (see Lemmas \ref{lem:3.1} and \ref{lem:3.2} for details).

Let us do a convergence analysis of Algorithm \ref{algo:1}.

\begin{thm}\label{thm:1}
Under Assumptions \ref{assum:1}--\ref{assum:5},
the sequence $(x_i(k))_{k\geq 0}$ ($i\in V$) generated by Algorithm \ref{algo:1} 
converges almost surely to a random point $x^\star \in X^\star$.
\end{thm}

Let us compare the distributed random projected gradient algorithm \cite[(2a) and (2b)]{lee2013} with Algorithm \ref{algo:1}.
The algorithm \cite[(2a) and (2b)]{lee2013} is the pioneering distributed optimization algorithm that is based on local communications 
of users' estimates in a network and a gradient descent with random projections.
It can be applied to Problem \ref{prob:1} 
when $f_i$ $(i\in V)$ is convex and differentiable with the Lipschitz gradient $\nabla f_i$ \cite[Assumption 1 c)]{lee2013}.
The algorithm \cite[(2a) and (2b)]{lee2013} is 
\begin{align}\label{lee}
x_i\left( k+1\right) := P_{X_i^{\Omega_i\left(k\right)}} \left(v_i\left(k\right) - \alpha_k \nabla f_i\left(v_i \left(k\right) \right)  \right),
\end{align}
where $v_i(k)$ is defined as in \eqref{viprox}. 
Proposition 1 in \cite{lee2013} indicates that, under the assumptions in Theorem \ref{thm:2}, 
the sequence $(x_i(k))_{k\geq 0}$ $(i\in V)$ generated by algorithm \eqref{lee}
converges almost surely to $x^\star \in X^\star$.
In contrast to algorithm \eqref{lee}, Algorithm \ref{algo:1} can be applied to nonsmooth convex optimization (see Assumption \ref{assum:1}(A2))
by using the subgradients $g_i(k) \in \partial f_i(v_i(k))$ $(i\in V, k\geq 0)$
and enables all users to arrive at the same solution to Problem \ref{prob:1} under the assumptions in Theorem \ref{thm:1}
that are stronger than the ones in Theorem \ref{thm:2}.
In Algorithm \ref{algo:2}, each user sets $x_i(k+1)$ by using its proximity operator (see \eqref{xiprox} for definition of $x_i(k)$ $(k\geq 0)$ in Algorithm \ref{algo:2}) and can solve Problem \ref{prob:1} under the assumptions in Theorem \ref{thm:2} (see Theorem \ref{thm:2}).

\subsection{Proof of Theorem \ref{thm:1}}\label{subsec:3.1}
Let us first show the following lemma.

\begin{lem}\label{lem:3.1}
Suppose that Assumption \ref{assum:1} holds. 
Then, for all $x\in X$, for all $i\in V$, for all $z\in X_i$, for all $k\geq 0$, and for all $\tau, \eta > 0$,
\begin{align*}
\left\| x_i \left(k+1\right) -x  \right\|^2 
&\leq 
\left\| v_i \left(k\right) -x  \right\|^2 -2 \alpha_k \left(f_i\left(z\right) -f_i\left(x\right) \right) 
+ M\left(\tau, \eta\right)\alpha_k^2\\
&\quad + \tau \left\| v_i\left(k\right) - z \right\|^2 + \left( \eta -1 \right) \left\| v_i\left(k\right) - x_i\left(k+1\right)\right\|^2,
\end{align*}
where $M(\tau,\eta)$ is a positive number depending on $\tau$ and $\eta$.
\end{lem}

\begin{proof}
Choose $x\in X$ and $i\in V$ arbitrarily and fix an arbitrary $k\geq 0$.
From the firm nonexpansivity of $P_{X_i^{\Omega_i(k)}}$ and $x = P_{X_i^{\Omega_i(k)}}(x)$,
\begin{align*}
\left\| x_i \left(k+1\right) -x  \right\|^2 
&= 
\left\| P_{X_i^{\Omega_i\left(k\right)}} \left(v_i\left(k\right) - \alpha_k g_i\left(k\right)  \right) 
  - P_{X_i^{\Omega_i\left(k\right)}}\left(x\right) \right\|^2\\
&\leq \left\| \left(v_i\left(k\right) -x \right) - \alpha_k g_i\left(k\right) \right\|^2 
  - \left\| \left(v_i\left(k\right) - x_i \left(k+1\right) \right) - \alpha_k g_i\left(k\right)   \right\|^2, 
\end{align*}
which, together with $\|x+y\|^2 = \|x\|^2 + 2\langle x,y\rangle + \|y\|^2$ $(x,y\in \mathbb{R}^d)$, means that 
\begin{align}
\left\| x_i \left(k+1\right) -x  \right\|^2
&\leq \left\|v_i\left(k\right) -x \right\|^2 - 2 \alpha_k \langle v_i\left(k\right) -x, g_i\left(k\right) \rangle 
 - \left\| v_i\left(k\right) - x_i \left(k+1\right)\right\|^2\nonumber\\
&\quad + 2 \alpha_k \langle v_i\left(k\right) -x_i \left(k+1\right), g_i\left(k\right) \rangle.\label{ineq:1} 
\end{align}
Choose $z\in X_i$ arbitrarily and set $g_i(z) \in \partial f_i(z)$. 
Then, the definition of $\partial f_i$ ensures that
\begin{align*}
2 \alpha_k \langle v_i\left(k\right) -x, g_i\left(k\right) \rangle
&\geq 
2 \alpha_k \left( f_i \left(v_i\left(k\right) \right) - f_i \left(x\right) \right)\\
&= 2 \alpha_k \left( f_i \left(z \right) - f_i \left(x\right) \right) 
    + 2 \alpha_k \left( f_i \left(v_i\left(k\right) \right) - f_i \left(z\right) \right)\\
&\geq 2 \alpha_k \left( f_i \left(z \right) - f_i \left(x\right) \right)
    + 2 \alpha_k \left\langle v_i\left(k\right) - z, g_i\left(z\right)  \right\rangle,            
\end{align*}
which, together with the Cauchy-Schwarz inequality and $2|a||b| \leq \tau a^2 + (1/\tau)b^2$ $(a,b\in \mathbb{R},\tau > 0)$
\cite[Inequality (8)]{nedic2011}, implies that, for all $\tau > 0$,
\begin{align}\label{ineq:2}
\begin{split}
2 \alpha_k \langle v_i\left(k\right) -x, g_i\left(k\right) \rangle
&\geq 2 \alpha_k \left( f_i \left(z \right) - f_i \left(x\right) \right)
       - \tau \left\|v_i\left(k\right) - z \right\|^2 - \frac{\alpha_k^2 \left\| g_i\left(z\right) \right\|^2}{\tau}.
\end{split}              
\end{align}
The Cauchy-Schwarz inequality and $2|a||b| \leq \tau a^2 + (1/\tau)b^2$ $(a,b\in \mathbb{R},\tau > 0)$
also mean that, for all $\eta > 0$,
\begin{align}\label{ineq:3}
2 \alpha_k \langle v_i\left(k\right) -x_i \left(k+1\right), g_i\left(k\right) \rangle
\leq \eta \left\| v_i\left(k\right) -x_i \left(k+1\right)\right\|^2 + \frac{\alpha_k^2 \left\|g_i\left(k\right)\right\|^2}{\eta}.
\end{align}
Inequalities \eqref{ineq:1}, \eqref{ineq:2}, and \eqref{ineq:3} imply that, for all $\tau,\eta > 0$,
\begin{align*}
\left\| x_i \left(k+1\right) -x  \right\|^2
&\leq 
\left\|v_i\left(k\right) -x \right\|^2
-2 \alpha_k \left( f_i \left(z \right) - f_i \left(x\right) \right)
       + \tau \left\|v_i\left(k\right) - z \right\|^2\\
&\quad +\left(\eta-1\right) \left\| v_i\left(k\right) - x_i \left(k+1\right)\right\|^2 
   + \left( \frac{\left\| g_i\left(z\right) \right\|^2}{\tau} + \frac{\left\|g_i\left(k\right)\right\|^2}{\eta}\right)\alpha_k^2.
\end{align*}
Since (A3)' ensures that $M (\tau,\eta) := \max_{i\in V} (\sup \{ \| g_i(z)\|^2/\tau + \|g_i(k)\|^2/\eta \colon z\in X_i, k\geq 0 \}) \leq \max_{i\in V}(M_i^2/\tau + C_i^2/\eta) < \infty$, the above inequality completes the proof of Lemma \ref{lem:3.1}.
\end{proof}

Lemma \ref{lem:3.1} leads to the following.

\begin{lem}\label{lem:3.2}
Suppose that Assumptions \ref{assum:1}, \ref{assum:3}, and \ref{assum:5} hold. 
Then, $\sum_{k=0}^\infty \mathrm{d}(v_i(k), X)^2 < \infty$ almost surely for all $i\in V$.
\end{lem}

\begin{proof}
Putting $x = z := P_X (v_i(k))$ $(i\in V, k\geq 0)$ in Lemma \ref{lem:3.1} implies that, 
for all $i\in V$, for all $k\geq 0$, and for all $\tau,\eta > 0$,
\begin{align*}
\left\| x_i \left(k+1\right) - P_X \left(v_i\left(k\right)\right)  \right\|^2 
&\leq 
\left\| v_i \left(k\right) - P_X \left(v_i\left(k\right)\right)  \right\|^2 + \tau \left\| v_i\left(k\right) - P_X \left(v_i\left(k\right)\right) \right\|^2\\
&\quad + \left( \eta -1 \right) \left\| v_i\left(k\right) - x_i\left(k+1\right)\right\|^2 + M\left(\tau, \eta\right)\alpha_k^2.
\end{align*}
The definition of $P_X$ means that $\mathrm{d}(v_i(k),X) = \| v_i(k) - P_X (v_i(k)) \|$ and 
$\mathrm{d}(x_i(k+1), X) \leq \| x_i (k+1) - P_X (v_i(k))\|$ $(i\in V, k\geq 0)$,
Hence, for all $i\in V$, for all $k\geq 0$, and for all $\tau,\eta > 0$,
\begin{align*}
\mathrm{d}\left(x_i\left(k+1\right), X\right)^2 
&\leq 
\mathrm{d}\left(v_i\left(k\right), X\right)^2 + \tau \mathrm{d}\left(v_i\left(k\right), X\right)^2
   + \left( \eta -1 \right) \left\| v_i\left(k\right) - x_i\left(k+1\right)\right\|^2\\
&\quad + M\left(\tau, \eta\right)\alpha_k^2.    
\end{align*}
Moreover, since \eqref{xi} means that $x_i(k+1) \in X_i^{\Omega_i(k)}$ $(i\in V,k\geq 0)$, 
that $\mathrm{d}(v_i(k), X_i^{\Omega_i(k)}) \leq \| v_i(k) - x_i(k+1)\|$ $(i\in V,k\geq 0)$ is found.
Accordingly, for all $i\in V$, for all $k\geq 0$, for all $\tau > 0$, and for all $\eta \in (0,1)$,
\begin{align*}
\mathrm{d}\left(x_i\left(k+1\right), X\right)^2 
&\leq 
\mathrm{d}\left(v_i\left(k\right), X\right)^2 + \tau \mathrm{d}\left(v_i\left(k\right), X\right)^2
   + \left( \eta -1 \right) \mathrm{d}\left(v_i\left(k\right), X_i^{\Omega_i(k)} \right)^2 \\
&\quad + M\left(\tau, \eta\right)\alpha_k^2.    
\end{align*}
Taking the expectation in this inequality conditioned on $\mathcal{F}_k$ defined in \eqref{F} leads to the finding that, 
for all $i\in V$, for all $k\geq 0$, for all $\tau > 0$, and for all $\eta \in (0,1)$, almost surely
\begin{align*}
\mathrm{E} \left[ \mathrm{d}\left(x_i\left(k+1\right), X\right)^2 \Big| \mathcal{F}_k \right]
&\leq 
\mathrm{d}\left(v_i\left(k\right), X\right)^2 + \tau \mathrm{d}\left(v_i\left(k\right), X\right)^2
   + M\left(\tau, \eta\right)\alpha_k^2\\
&\quad + \left( \eta -1 \right) \mathrm{E} \left[ \mathrm{d}\left(v_i\left(k\right), X_i^{\Omega_i(k)} \right)^2 \Big| \mathcal{F}_k \right], 
\end{align*}
which, together with Assumption \ref{assum:3}, implies that, almost surely
for all $i\in V$, for all $k\geq 0$, for all $\tau > 0$, and for all $\eta \in (0,1)$, 
\begin{align*}
\mathrm{E} \left[ \mathrm{d}\left(x_i\left(k+1\right), X\right)^2 \Big| \mathcal{F}_k \right]
&\leq \mathrm{d}\left(v_i\left(k\right), X\right)^2 + \left( \tau  + \frac{\eta -1}{c} \right)\mathrm{d}\left(v_i\left(k\right), X \right)^2
  + M\left(\tau, \eta\right)\alpha_k^2.
\end{align*}
Here, let us take $\tau := 1/(2c)$ and $\eta := 1/4$. 
From $\tau  + (\eta -1)/c = -1/(4c)$ and the convexity of $\mathrm{d}(\cdot,X)^2$, 
almost surely for all $i\in V$ and for all $k\geq 0$,
\begin{align*}
\mathrm{E} \left[ \mathrm{d}\left(x_i\left(k+1\right), X\right)^2 \Big| \mathcal{F}_k \right]
&\leq 
\sum_{j=1}^m \left[W(k)\right]_{ij} \mathrm{d}\left(x_j\left(k\right), X\right)^2 - \frac{1}{4c} \mathrm{d}\left(v_i\left(k\right), X \right)^2\\
&\quad + M\left(\frac{1}{2c}, \frac{1}{4} \right)\alpha_k^2.
\end{align*}
Hence, Assumption \ref{assum:5} ensures that, almost surely for all $k\geq 0$, 
\begin{align}\label{key3}
\begin{split}
\mathrm{E} \left[ \sum_{i=1}^m \mathrm{d}\left(x_i\left(k+1\right), X\right)^2 \Bigg| \mathcal{F}_k \right]
&\leq 
\sum_{j=1}^m  \mathrm{d}\left(x_j\left(k\right), X\right)^2 
 - \frac{1}{4c} \sum_{i=1}^m \mathrm{d}\left(v_i\left(k\right), X \right)^2\\
&\quad + mM\left(\frac{1}{2c}, \frac{1}{4} \right)\alpha_k^2.
\end{split}
\end{align}
Proposition \ref{prop:2} and (C2) lead to $\sum_{k=0}^\infty \sum_{i=1}^m \mathrm{d}(v_i(k), X)^2 < \infty$ almost surely; i.e.,
$\sum_{k=0}^\infty \mathrm{d}(v_i(k), X)^2 < \infty$ almost surely for all $i\in V$.
This completes the proof.
\end{proof}

A discussion similar to the one for proving Lemma \ref{lem:4.4} leads to the following.

\begin{lem}\label{lem:3.4}
Suppose that Assumptions \ref{assum:1}, \ref{assum:4}, and \ref{assum:5} hold, and 
define $\bar{v}(k) := (1/m) \sum_{l=1}^m v_l(k)$ and $e_i(k) := x_i(k+1) - v_i(k)$ for all $k\geq 0$ and for all $i\in V$.
Then, $\sum_{k=0}^\infty \| e_i(k)\|^2 < \infty$ and $\sum_{k=0}^\infty \alpha_k \| v_i(k) - \bar{v}(k)\| < \infty$
almost surely for all $i\in V$.
\end{lem}

\begin{proof}
Define $z_i(k) := P_X (v_i(k))$ $(i\in V,k\geq 0)$.
From the triangle inequality, $z_i(k) \in X \subset X_i^{\Omega_i(k)}$,
and the nonexpansivity of $P_{X_i^{\Omega_i(k)}}$ $(i\in V,k\geq 0)$,
\begin{align*}
\left\| e_i\left(k\right) \right\| 
&\leq 
\left\| x_i\left(k+1\right) - z_i\left(k\right)\right\| + \left\| z_i\left(k\right) - v_i\left(k\right)\right\|\\
&= \left\| P_{X_i^{\Omega_i\left(k\right)}} \left(v_i\left(k\right) - \alpha_k g_i\left(k\right)  \right) 
    - P_{X_i^{\Omega_i\left(k\right)}} \left( z_i\left(k\right) \right) \right\|
    + \left\| z_i\left(k\right) - v_i\left(k\right)\right\|\\
&\leq  \left\| \left(v_i\left(k\right) - z_i\left(k\right) \right) - \alpha_k g_i\left(k\right) \right\| 
    + \left\| z_i\left(k\right) - v_i\left(k\right)\right\|.  
\end{align*}
Hence, the triangle inequality and (A3)' ensure that, for all $i\in V$ and for all $k\geq 0$,
$\| e_i (k) \|
\leq 2 \| z_i (k) - v_i(k) \| + C_i \alpha_k$,
which, together with $(a+b)^2 \leq 2(a^2 + b^2)$ $(a,b\in\mathbb{R})$ and the definition of $z_i(k)$, 
implies that, for all $i\in V$ and for all $k\geq 0$,
\begin{align*}
\left\| e_i\left(k\right) \right\|^2
\leq 4 \left\| z_i\left(k\right) - v_i\left(k\right)\right\|^2 + 2 C_i^2 \alpha_k^2 
= 4 \mathrm{d}\left( v_i\left(k\right), X  \right)^2 + 2 C_i^2 \alpha_k^2.
\end{align*}
Accordingly, Lemma \ref{lem:3.2} and (C2) lead to $\sum_{k=0}^\infty \| e_i(k)\|^2 < \infty$
almost surely for all $i\in V$.
The definition of $v_i(k)$ $(i\in V, k\geq 0)$ in Algorithm \ref{algo:1} is the same as in \eqref{viprox}.
Therefore, a discussion similar to the one for proving Lemma \ref{lem:4.4} leads to $\sum_{k=0}^\infty \alpha_k \| v_i(k) - \bar{v}(k)\| < \infty$
almost surely for all $i\in V$. 
This completes the proof.
\end{proof}

\begin{lem}\label{lem:3.5}
Suppose that the assumptions in Theorem \ref{thm:1} are satisfied and define $z_i(k) := P_X(v_i(k))$ for all $i\in V$ and for all $k\geq 0$ and 
$\bar{z}(k) := (1/m) \sum_{i=1}^m z_i(k)$ for all $k\geq 0$.
Then, the sequence $(\| x_i(k) - x^\star \|)_{k\geq 0}$ converges almost surely for all $i\in V$ and for all $x^\star \in X^\star$
and $\liminf_{k\to\infty} f(\bar{z}(k)) = f^\star$ almost surely.
\end{lem}

\begin{proof}
Choose $x^\star \in X^\star$ arbitrarily.
The convexity of $\|\cdot\|^2$ and Assumption \ref{assum:5} imply that 
$\sum_{i=1}^m \| v_i(k) - x^\star\|^2 \leq \sum_{i=1}^m \| x_i(k) - x^\star\|^2$ $(k\geq 0)$.
Summing the inequality in Lemma \ref{lem:3.1} over all $i$ guarantees that, for all $k\geq 0$ and for all $\tau, \eta > 0$,
\begin{align}
\sum_{i=1}^m \left\| x_i \left(k+1\right) -x^\star  \right\|^2 
&\leq 
\sum_{i=1}^m \left\| x_i \left(k\right) -x^\star  \right\|^2 -2 \alpha_k \sum_{i=1}^m \left(f_i\left(z_i\left(k\right) \right) -f_i\left(x^\star \right) \right)\nonumber\\
&\quad + \tau \sum_{i=1}^m \left\| v_i\left(k\right) - z_i\left(k\right) \right\|^2 
       + \left( \eta -1 \right) \sum_{i=1}^m \left\| v_i\left(k\right) - x_i\left(k+1\right)\right\|^2\nonumber\\
&\quad + m M\left(\tau, \eta\right)\alpha_k^2.\label{ineq:3.5.1}
\end{align}
It can be observed that \eqref{fi} holds for Algorithm \ref{algo:1} because the definitions of $\bar{z}(k)$ and $\bar{v}(k)$ $(k\geq 0)$ in the proof of Theorem \ref{thm:2} are the same as the definitions of $\bar{z}(k)$ and $\bar{v}(k)$ $(k\geq 0)$ in Lemma \ref{lem:3.5}. Therefore, the definitions of $f$ and $f^\star$ imply that, for all $k\geq 0$ and for all $\tau,\eta > 0$,
\begin{align*}
\sum_{i=1}^m \left\| x_i \left(k+1\right) -x^\star  \right\|^2 
&\leq 
\sum_{i=1}^m \left\| x_i \left(k\right) -x^\star  \right\|^2 
+ 2 \bar{M} \alpha_k \sum_{i=1}^m \left\| v_i\left(k\right) - \bar{v}\left(k\right) \right\|\\
&\quad + 2 \bar{M} \alpha_k \sum_{l=1}^m \left\| v_l\left(k\right) - \bar{v}\left(k\right) \right\|
-2 \alpha_k \left(f \left(\bar{z}\left(k\right) \right) -f^\star \right)\\
&\quad + \tau \sum_{i=1}^m \left\| v_i\left(k\right) - z_i\left(k\right) \right\|^2 
       + \left( \eta -1 \right) \sum_{i=1}^m \left\| v_i\left(k\right) - x_i\left(k+1\right)\right\|^2\\
&\quad + m M\left(\tau, \eta\right)\alpha_k^2.
\end{align*}
From $\mathrm{d}(v_i(k),X) = \| v_i(k) - z_i(k) \|$ and 
$\mathrm{d}(v_i(k), X_i^{\Omega_i(k)}) \leq \| v_i(k) - x_i(k+1)\|$ $(i\in V,k\geq 0)$,
for all $k\geq 0$, for all $\tau > 0$, and for all $\eta \in (0,1)$,
\begin{align*}
\sum_{i=1}^m \left\| x_i \left(k+1\right) -x^\star  \right\|^2
&\leq 
\sum_{i=1}^m \left\| x_i \left(k\right) -x^\star  \right\|^2 
+ 4 \bar{M} \alpha_k \sum_{i=1}^m \left\| v_i\left(k\right) - \bar{v}\left(k\right) \right\|\\
&\quad 
-2 \alpha_k \left(f \left(\bar{z}\left(k\right) \right) -f^\star \right) 
+ \tau \sum_{i=1}^m \mathrm{d} \left(v_i\left(k\right), X \right)^2\\ 
&\quad + \left( \eta -1 \right) \sum_{i=1}^m \mathrm{d} \left( v_i\left(k\right), X_i^{\Omega_i\left(k\right)}\right)^2
  + m M\left(\tau, \eta\right)\alpha_k^2. 
\end{align*}
Hence, Assumption \ref{assum:3} guarantees that, almost surely
for all $k\geq 0$, for all $\tau > 0$, and for all $\eta \in (0,1)$,
\begin{align}
\mathrm{E} \left[\sum_{i=1}^m \left\| x_i \left(k+1\right) -x^\star  \right\|^2 \Bigg| \mathcal{F}_k \right]
&\leq 
\sum_{i=1}^m \left\| x_i \left(k\right) -x^\star  \right\|^2 
+ 4 \bar{M} \alpha_k \sum_{i=1}^m \left\| v_i\left(k\right) - \bar{v}\left(k\right) \right\|\nonumber\\
&\quad 
-2 \alpha_k \left(f \left(\bar{z}\left(k\right) \right) -f^\star \right) 
+ \tau \sum_{i=1}^m \mathrm{d} \left(v_i\left(k\right), X \right)^2\nonumber\\ 
&\quad 
+ \frac{\eta -1}{c} \sum_{i=1}^m \mathrm{d} \left( v_i\left(k\right), X \right)^2
+ m M\left(\tau, \eta\right)\alpha_k^2,\label{key4} 
\end{align}
which, together with $\tau := 1/(2c)$, $\eta := 1/4$, and $\tau + (\eta -1)/c = -1/(4c)$, implies that, almost surely for all $k\geq 0$,
\begin{align*}
\mathrm{E} \left[\sum_{i=1}^m \left\| x_i \left(k+1\right) -x^\star  \right\|^2 \Bigg| \mathcal{F}_k \right]
&\leq 
\sum_{i=1}^m \left\| x_i \left(k\right) -x^\star  \right\|^2
-2 \alpha_k \left(f \left(\bar{z}\left(k\right) \right) -f^\star \right)\\ 
&\quad + 4 \bar{M} \alpha_k \sum_{i=1}^m \left\| v_i\left(k\right) - \bar{v}\left(k\right) \right\|
+ m M\left(\frac{1}{2c}, \frac{1}{4}\right)\alpha_k^2,
\end{align*}
where $M (1/(2c), 1/4) < \infty$ holds from (A3)'.
Therefore, Proposition \ref{prop:2}, (C2), and Lemma \ref{lem:3.4} ensure that 
$(\sum_{i=1}^m \| x_i(k) -x^\star \|^2)_{k\geq 0}$ converges almost surely; i.e., 
$(\| x_i(k) -x^\star \|)_{k\geq 0}$ converges almost surely for all $i\in V$.
Moreover, since $\bar{z}(k) \in X$ implies $f(\bar{z}(k)) - f^\star \geq 0$ $(k\geq 0)$,
there is also another finding: almost surely
\begin{align}\label{f}
\sum_{k=0}^\infty \alpha_k \left(f \left(\bar{z}\left(k\right) \right) -f^\star \right) < \infty.
\end{align}
Hence, a discussion similar to the one for proving Lemma \ref{lem:4.5} leads to
$\liminf_{k\to\infty} f(\bar{z}(k)) = f^\star$ almost surely.
This completes the proof.
\end{proof}

We can prove Theorem \ref{thm:1} by referring to the proof of Theorem \ref{thm:2}.

\begin{proof}
By referring to the proof of Theorem \ref{thm:2}, 
the almost sure convergence of $(x_i(k))_{k\geq 0}$ $(i\in V)$ and Lemma \ref{lem:3.2} lead to the conclusion that
$(v_i(k))_{k\geq 0}$, $(z_i(k))_{k\geq 0}$ $(i\in V)$, $(\bar{v}(k))_{k\geq 0}$, and $(\bar{z}(k))_{k\geq 0}$ converge almost surely.
Moreover, Lemma \ref{lem:3.5} and the continuity of $f$ ensure that 
$(\bar{v}(k))_{k\geq 0}$ and $(\bar{z}(k))_{k\geq 0}$ converge almost surely to $x_* \in X^\star$.
By referring to the proof of Theorem \ref{thm:2},
Lemma \ref{lem:3.4}, (C1), and the triangle inequality 
guarantee that $\lim_{k\to\infty} \| v_i(k) - x_*\| = 0$ almost surely for all $i\in V$.
Since Lemma \ref{lem:3.4} ensures that, for all $i\in V$, $\lim_{k\to\infty} \|e_i(k)\|^2 = \lim_{k\to\infty} \| x_i(k+1) - v_i(k) \|^2 = 0$ almost surely,
$(x_i(k))_{k\geq 0}$ $(i\in V)$ converges almost surely to $x_* \in X^\star$.
This completes the proof.
\end{proof}

\subsection{Convergence rate analysis for Algorithm \ref{algo:1}}\label{subsec:3.2}
The following proposition is proven by referring to the discussion in subsection \ref{subsec:3.1}.

\begin{prop}\label{prop:thm:1}
Suppose that the assumptions in Theorem \ref{thm:1} hold, $x^\star \in X^\star$ is a solution to Problem \ref{prob:1}, 
and $(x_i(k))_{k\geq 0}$ ($i\in V$) is the sequence generated by Algorithm \ref{algo:1}. 
Then, there exist $\beta^{(j)} > 0$ ($j=1,2,3,4$) such that, almost surely for all $k\geq 0$,
\begin{align*}
&\mathrm{E} \left[ \sum_{i=1}^m \mathrm{d} \left( x_i\left(k+1\right), X  \right)^2 \Bigg| \mathcal{F}_k  \right]
\leq  \sum_{i=1}^m \mathrm{d} \left( x_i\left(k\right), X  \right)^2 
      - \beta^{(1)} \gamma_k^{(1)} 
      + \beta^{(2)} \gamma_k^{(2)},\\
&\mathrm{E} \left[ \sum_{i=1}^m \left\| x_i\left(k+1\right) - x^\star  \right\|^2 \Bigg| \mathcal{F}_k  \right]
\leq  \sum_{i=1}^m \left\| x_i\left(k\right) - x^\star  \right\|^2 
      - \sum_{j=1,4} \beta^{(j)} \gamma_k^{(j)} 
      + \sum_{j=2,3} \beta^{(j)} \gamma_k^{(j)},
\end{align*}
where $\gamma_k^{(1)} := \sum_{i=1}^m \mathrm{d}(v_i(k),X)^2$,
$\gamma_k^{(2)} := \alpha_k^2$,
$\gamma_k^{(3)} := \alpha_k \sum_{i=1}^m \|v_i(k) - \bar{v}(k)  \|$, and
$\gamma_k^{(4)} := \alpha_k (f(\bar{z}(k)) -f^\star)$ ($k\geq 0$) satisfy $\sum_{k=0}^\infty \gamma_k^{(j)} < \infty$
($j=1,2,3,4$).
\end{prop}

\begin{proof}
From Lemma \ref{lem:3.2} and (C2),  
$\gamma_k^{(1)} := \sum_{i=1}^m \mathrm{d}(v_i(k),X)^2$ and $\gamma_k^{(2)} := \alpha_k^2$ $(k\geq 0)$ satisfy $\sum_{k=0}^\infty \gamma_k^{(1)} < \infty$ almost surely and $\sum_{k=0}^\infty \gamma_k^{(2)} < \infty$.
Lemma \ref{lem:3.4} and \eqref{f} ensure that  
$\gamma_k^{(3)} := \alpha_k \sum_{i=1}^m \|v_i(k) - \bar{v}(k)  \|$ and
$\gamma_k^{(4)} := \alpha_k (f(\bar{z}(k)) -f^\star)$ $(k\geq 0)$ 
also satisfy 
$\sum_{k=0}^\infty \gamma_k^{(j)} < \infty$ almost surely for $j=3,4$.
Put $\tau := 1/(2c)$ and $\eta := 1/4$, $\beta^{(1)} := - (\tau + (\eta-1)/c) =  1/(4c)$, 
$\beta^{(2)} := m M(\tau, \eta)$,
$\beta^{(3)} := 4 \bar{M}$, 
and 
$\beta^{(4)} := 2$,
where $\beta^{(2)}, \beta^{(3)} < \infty$ hold from (A3)' and $\bar{M} := \max_{i\in V} M_i$.
Accordingly, \eqref{key3} and \eqref{key4} ensure that Proposition \ref{prop:thm:1} holds.
\end{proof}

A discussion similar to the one for obtaining \eqref{rate}, 
Proposition \ref{prop:thm:1}, and Theorem \ref{thm:1} indicate that $(x_i(k))_{k\geq 0}$ $(i\in V)$ in Algorithm \ref{algo:1} converges almost surely to a solution to Problem \ref{prob:1} under the following convergence rates:
almost surely, for all $k\geq 0$,
\begin{align}\label{rate1}
\begin{split}
&\mathrm{E} \left[ \sum_{i=1}^m \mathrm{d} \left( x_i\left(k+1\right), X  \right)^2 \Bigg| \mathcal{F}_k  \right]
\leq \sum_{i=1}^m \mathrm{d} \left( x_i\left(k\right), X  \right)^2 
+ O\left(\alpha_k^2\right),\\
&\mathrm{E} \left[ \sum_{i=1}^m \left\| x_i\left(k+1\right) - x^\star  \right\|^2 \Bigg| \mathcal{F}_k  \right]
\leq  \sum_{i=1}^m \left\| x_i\left(k\right) - x^\star  \right\|^2 
+ O \left(\alpha_k\right).
\end{split}
\end{align}
It can be observed from \eqref{rate} and \eqref{rate1} that Algorithms \ref{algo:2} and \ref{algo:1} have almost the same convergence rate.
Section \ref{sec:5} gives numerical examples such that Algorithms \ref{algo:2} and \ref{algo:1} have almost the same convergence rate
when they have the same step size sequences.

Next, let us consider the case where $f_i$ $(i\in V)$ is convex and differentiable and $\nabla f_i$ $(i\in V)$ satisfies the Lipschitz continuity condition \cite[Assumption 1 c)]{lee2013}. Then, Algorithm \ref{algo:1} coincides with the first distributed random projection algorithm \cite[(2a) and (2b)]{lee2013} (see \eqref{lee}) defined as follows for all $k\geq 0$ and for all $i\in V$:
\begin{align}\label{lee1}
x_i\left( k+1\right) := P_{X_i^{\Omega_i\left(k\right)}} \left(v_i\left(k\right) - \alpha_k \nabla f_i\left(v_i \left(k\right) \right)  \right),
\end{align}
where $v_i(k)$ $(i\in V,k\geq 0)$ is as in \eqref{viprox} and $(\alpha_k)_{k\geq 0}$ satisfies (C1) and (C2).
The proof of Lemma 5 and (18) in \cite{lee2013} indicate that 
algorithm \eqref{lee1}, under (A3), almost surely satisfies the following convergence rates: for a large enough $k$,
\begin{align}\label{rate2}
\begin{split}
&\mathrm{E} \left[ \sum_{i=1}^m \mathrm{d} \left( x_i\left(k+1\right), X  \right)^2 \Bigg| \mathcal{F}_k  \right]
\leq \left( 1 + O \left(\alpha_k^2 \right) \right) \sum_{i=1}^m \mathrm{d} \left( x_i\left(k\right), X  \right)^2 
+ O\left(\alpha_k^2\right),\\
&\mathrm{E} \left[ \sum_{i=1}^m \left\| x_i\left(k+1\right) - x^\star  \right\|^2 \Bigg| \mathcal{F}_k  \right]
\leq  \left( 1 + O \left( \alpha_k^2 \right) \right) \sum_{i=1}^m \left\| x_i\left(k\right) - x^\star  \right\|^2 
+ O \left(\alpha_k\right).
\end{split}
\end{align}
Proposition \ref{prop:thm:1} implies that, if stronger assumption (A3)' is satisfied, 
algorithm \eqref{lee1} satisfies \eqref{rate1} that are better properties for convergence rate than \eqref{rate2}.

\section{Experimental results}\label{sec:5}
Let us apply Algorithms \ref{algo:2} and \ref{algo:1} to Problem \ref{prob:1} with $f_i \colon \mathbb{R}^d \to \mathbb{R}$,
and $X_i \subset \mathbb{R}^d$ $(i\in V := \{ 1,2,\ldots,m \})$ defined by 
\begin{align*}
f_i \left( x  \right) := \sum_{j=1}^d a_{ij} \left| x_j - b_{ij} \right| \text{ and } 
X_i := \bigcap_{j=1}^d \left\{ x\in \mathbb{R}^d \colon \left\| x- c_{ij} \right\| \leq r_{ij} \right\},
\end{align*}
where $a_i := (a_{ij})_{j=1}^d, b_i := (b_{ij})_{j=1}^d, r_i := (r_{ij})_{j=1}^d \in \mathbb{R}_+^d$
$(i\in V)$, and $c_{ij} \in \mathbb{R}^d$ $(i\in V, j=1,2,\ldots,d)$,
i.e., the problem of minimizing the sum of the weighted $L^1$-norms 
$f(x) = \sum_{i=1}^m f_i(x)$ 
over the intersection of closed balls  
$X = \bigcap_{i=1}^m X_i$.
The metric projection onto $X_i^j := \{ x\in \mathbb{R}^d \colon \| x- c_{ij} \| \leq r_{ij} \}$ $(i\in V, j =1,2,\ldots,d)$
can be computed within a finite number of arithmetic operations \cite[Chapter 28]{b-c}.
The function $f_i$ $(i\in V)$ satisfies the Lipschitz continuity condition.
Hence, Assumptions \ref{assum:1} and \ref{assum:3} hold.
The set $X_i^{\Omega_i(k)}$ $(i\in V, k\geq 0)$ used in the experiment was chosen randomly from the sets $X_i^j$ so as to satisfy Assumption \ref{assum:2}.
The subgradient $\partial f_i$ and the proximity operator $\mathrm{prox}_{\alpha f_i}$ $(i\in V, \alpha > 0)$ can be calculated explicitly 
\cite[Lemma 10, (30), (35)]{comb2007}.

\begin{figure}[H]
 \centering
  \includegraphics[width=5.5cm]{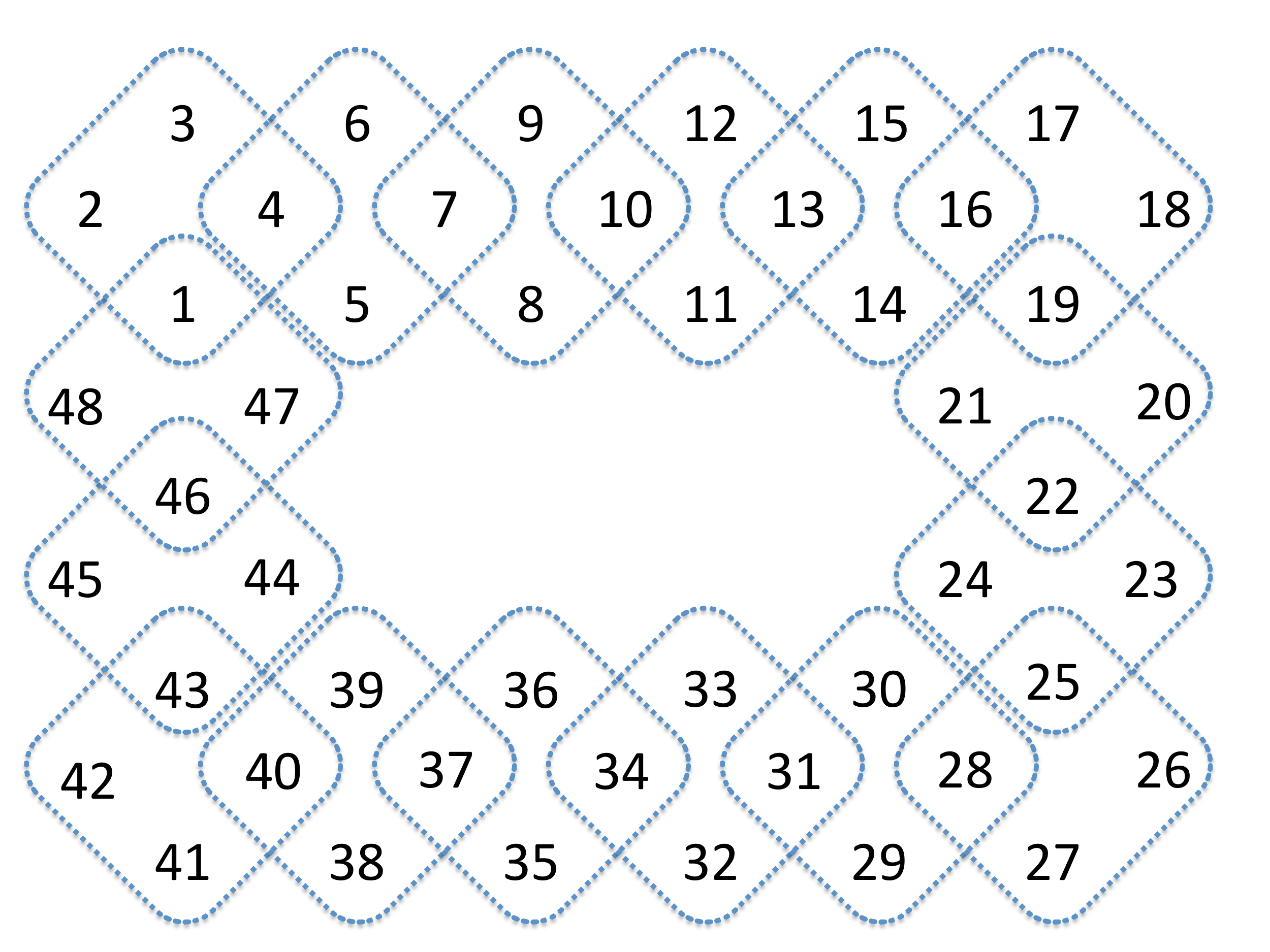}
  \caption{Network model used in experiment (Users within a dotted line can communicate with each other.)}\label{fig:1}
\end{figure}

In the experiment, we used a network with 48 users (i.e., $m:=48$) and 16 subnetworks, as illustrated in Figure \ref{fig:1}. It was assumed that users within a dotted line (all users in each subnetwork) can communicate with each other. For example, user 2 can communicate with users 1, 3, and 4 (i.e., $N_{2}(k) = \{1,2,3,4 \}$), while user 1 can communicate with not only users 2, 3, and 4 but also users 46, 47, and 48 (i.e., $N_1(k) = \{1,2,3,4,46,47,48 \}$). The weighted parameters $w_{ij}(k)$ $(i\in V, j\in N_i(k))$ were set to satisfy Assumption \ref{assum:5} (e.g., user 2 has $w_{2,j}$ $(j\in N_2(k))$ such that $w_{2,j}(k) = w_{j,2}(k) = 3/8$ $(j=2,3)$ and $w_{2,j}(k) = w_{j,2}(k) = 1/8$ $(j=1,4)$, and user 1 has $w_{1,j}$ $(j\in N_1(k))$ such that $w_{1,1}(k) = 2/8$ and $w_{1,j}(k) = w_{j,1} = 1/8$ $(j=2,3,4,46,47,48)$). The point $v_i(k)$ $(i\in V)$ defined in \eqref{viprox} (e.g., $v_2(k) = (3/8)(x_2(k) + x_3(k)) + (1/8)(x_1(k) + x_4(k))$) was computed by passing along $x_j(k)$ $(j\in N_i(k))$ in a prearranged cyclic order.

The computer used in the experiment had two Intel Xeon E5-2640 v3 (2.60 GHz) CPUs. One had 8 physical cores and 16 threads; i.e., the total number of physical cores was 16, and the total number of threads was 32. The computer had 64 GB DDR4 memory and the Ubuntu 14.04.1 (Linux kernel: 3.16.0-30-generic, 64 bit) operating system. The experimental programs were run in Python 3.4.0; Numpy 1.8.2 was used to compute the linear algebra operations. We set $d:= 100$ and $m:=48$ and used 
$a_{i}:= (a_{ij})_{j=1}^d \in (0,1]^d, b_{i}:= (b_{ij})_{j=1}^d \in [0,1)^d, r_{i}:= (r_{ij})_{j=1}^d \in [3,4)^d$ $(i\in V)$,
and
$c_{ij} \in [-\sqrt{(3/4)d}, \sqrt{(3/4)d})^d$ $(i\in V, j=1,2,\ldots,d)$ to satisfy $X \neq \emptyset$ generated randomly by \texttt{numpy.random}. We performed 100 samplings, each starting from different random initial points $x_i(0)$ $(i\in V)$ in the range of $[-2,2]^d$, and averaged the results.

From the discussion in subsections \ref{subsec:4.2} and \ref{subsec:3.2}, it can be expected that 
Algorithms \ref{algo:2} and \ref{algo:1} with small step size sequences have fast convergence.
To see how the choice of step size sequence affects the convergence rate of Algorithms \ref{algo:2} and \ref{algo:1},
we compared Algorithms \ref{algo:2} and \ref{algo:1} when 
$\alpha_k = 1/(k+1)$
with  Algorithms \ref{algo:2} and \ref{algo:1} when
$\alpha_k = 10^{-3}/(k+1)$.

Let us define two performance measures for each $i\in V$ and for all $k\geq 0$, 
\begin{align}\label{FD}
D_i \left(k\right) :=  \left\| x_i \left(k\right) - \prod_{j=1}^d P_{X_i^j} \left(x_i\left(k\right)\right) \right\|
\text{ and }
F_i \left(k\right) :=  f_i \left(x_i \left(k \right) \right),
\end{align}
and observe the behaviors of $D_i(k)$ and $F_i(k)$ for Algorithms \ref{algo:2} and \ref{algo:1} with $\alpha_k = 1/(k+1), 10^{-3}/(k+1)$. If $(D_i(k))_{k\geq 0}$ $(i\in V)$ converges to 0, $(x_i(k))_{k\geq 0}$ converges to a fixed point of $\prod_{j=1}^d P_{X_i^j}$, i.e., to a point in $X_i = \bigcap_{j=1}^d X_i^j$ \cite[Corollary 4.37]{b-c}.
We investigated the behaviors of $(D_i(k))_{k=0}^{1000}$ and $(F_i(k))_{k=0}^{1000}$ $(i=1,2,\ldots,48)$ and observed that $(x_i(k))_{k=0}^{1000}$ $(i=1,2,\ldots,48)$ generated by Algorithms \ref{algo:2} and \ref{algo:1} with $\alpha_k = 10^{-3}/(k+1)$ converge faster than $(x_i(k))_{k=0}^{1000}$ generated by Algorithms \ref{algo:2} and \ref{algo:1} with $\alpha_k = 1/(k+1)$ and that they have almost the same convergence rate when they use the same step size sequences.
We also observed that the sequences of all users generated by both Algorithms \ref{algo:2} and \ref{algo:1} converge to the same point. The details are omitted due to the lack of space.

Let us divide all users into 16 groups,
\begin{align*}
G_{1} := \left\{2,3,4 \right\}, \text{ } G_{2} := \left\{ 5,6,7 \right\}, \ldots,\text{ } G_{16} := \left\{47, 48, 1\right\},
\end{align*} 
and compare the values of the objective functions and constraint evaluations in each of the groups defined by
\begin{align*}
F_{G_{j}} := \sum_{i\in G_{j}} F_i\left(10^3\right) \text{ and } D_{G_{j}} := \sum_{i\in G_{j}} D_i\left(10^3\right) 
\text{ } \left(j=1,2,\ldots,16 \right),
\end{align*}
where $F_i(k)$ and $D_i(k)$ $(i\in V, k \geq 0)$ are defined as in \eqref{FD}.

Table \ref{table:1} shows the values of $F_{G_{j}}$ and $D_{G_{j}}$ $(j=1,2,\ldots,16)$ for Algorithm \ref{algo:2} when $\alpha_k =1/(k+1), 10^{-3}/(k+1)$ and that all $D_{G_{j}}$ generated by Algorithm \ref{algo:2} with $\alpha_k = 10^{-3}/(k+1)$ were smaller than all $D_{G_{j}}$ generated with $\alpha_k = 1/(k+1)$. 
In particular, 
$D_{G_{4}}$, $D_{G_{7}}$, $D_{G_{9}}$, $D_{G_{14}}$, and $D_{G_{16}}$ were dramatically lower with Algorithm \ref{algo:2} with $\alpha_k = 10^{-3}/(k+1)$. It can be seen from Table \ref{table:2} that Algorithm \ref{algo:1} with $\alpha_k = 10^{-3}/(k+1)$ performed better than with $\alpha_k = 1/(k+1)$. This is because $D_{G_{5}}$, $D_{G_{7}}$, $D_{G_{9}}$, $D_{G_{14}}$, and $D_{G_{16}}$ were approximately zero when $\alpha_k = 10^{-3}/(k+1)$ was used. Tables \ref{table:1} and \ref{table:2} also show that all $F_{G_{j}}$ generated by Algorithm \ref{algo:2} were almost the same as all $F_{G_{j}}$ generated by Algorithm \ref{algo:1}.

The analyses in subsections \ref{subsec:4.2} and \ref{subsec:3.2} and the results of the experiment indicate that the rate of convergence of Algorithm \ref{algo:2} is almost the same as that of Algorithm \ref{algo:1} when they have the same step size sequence and that the algorithms are stable and have fast convergence when they have small step size sequences.
 
\section{Conclusion and future work}\label{sec:6}
The problem of minimizing the sum of all users' nonsmooth convex objective functions over the intersection of all users' closed convex constraint sets was discussed, and two distributed algorithms were presented for solving the problem. One algorithm uses each user's proximity operator and metric projection onto a set randomly selected from components of its constraint set while the other is obtained by replacing the proximity operator of the first algorithm with a subgradient. Convergence analysis showed that, under certain assumptions, the sequences of all users generated by each of the two algorithms converge almost surely to the same solution to the problem. It also indicated that the rates of convergence depend on the step size sequences and that it is desired to use small step size sequences so that the algorithms have fast convergence. The results of numerical evaluation using a concrete nonsmooth convex optimization problem support this analysis and demonstrate the effectiveness of the two algorithms.

The proposed algorithms can work when each user randomly sets one metric projection selected from many projections. Since nonexpansive mappings are generalization of metric projections and thus have wider application, developing distributed random algorithms that work when one user randomly chooses one nonexpansive mapping at a time is a promising undertaking.

\section*{Acknowledgements}
I thank Kazuhiro Hishinuma for his input on the numerical examples.

\begin{table}[H]
  \centering
  \caption{Values of $F_{G_{j}} := \sum_{i\in G_{j}} F_i(10^3)$ and $D_{G_{j}} := \sum_{i\in G_{j}} D_i(10^3)$ 
  $(j=1,2,\ldots,16)$ for Algorithm \ref{algo:2} when $\alpha_k =1/(k+1), 10^{-3}/(k+1)$}
  \label{table:1}
  \subcaptionbox{$\alpha_k =1/(k+1)$}{\begin{tabular}{c|cc}
\hline
Group & $F_{G_{j}}$ & $D_{G_{j}}$ \\
\hline
01 & 163.428093 & 0.325050 \\
02 & 158.725168 & 0.301325 \\
03 & 154.834519 & 0.340681 \\
04 & 157.867759 & 0.325576 \\
05 & 169.889020 & 0.342258 \\
06 & 160.626418 & 0.304963 \\
07 & 161.166806 & 0.328647 \\
08 & 148.522159 & 0.342929 \\
09 & 160.980531 & 0.270875 \\
10 & 160.546361 & 0.356931 \\
11 & 162.681703 & 0.326232 \\
12 & 158.329823 & 0.361666 \\
13 & 166.747981 & 0.329523 \\
14 & 157.684287 & 0.253613 \\
15 & 154.678365 & 0.326188 \\
16 & 150.104832 & 0.292431 \\
\hline
\end{tabular}}
  \subcaptionbox{$\alpha_k =10^{-3}/(k+1)$}{\begin{tabular}{c|cc}
\hline
Group & $F_{G_{j}}$ & $D_{G_{j}}$ \\
\hline
01 & 166.805046 & 0.000154 \\
02 & 161.267547 & 0.000816 \\
03 & 157.564483 & 0.000780 \\
04 & 160.620520 & 0.000016 \\
05 & 172.271029 & 0.000024 \\
06 & 162.896308 & 0.006480 \\
07 & 163.472768 & 0.000000 \\
08 & 151.511558 & 0.000712 \\
09 & 163.481266 & 0.000000 \\
10 & 163.301971 & 0.000316 \\
11 & 164.912426 & 0.006555 \\
12 & 160.847494 & 0.022754 \\
13 & 169.300357 & 0.010475 \\
14 & 160.186324 & 0.000000 \\
15 & 156.894949 & 0.000785 \\
16 & 152.450221 & 0.000000 \\
\hline
\end{tabular}}
\end{table}

\begin{table}[H]
  \centering
  \caption{Values of $F_{G_{j}} := \sum_{i\in G_{j}} F_i(10^3)$ and $D_{G_{j}} := \sum_{i\in G_{j}} D_i(10^3)$ 
  $(j=1,2,\ldots,16)$ for Algorithm \ref{algo:1} when $\alpha_k =1/(k+1), 10^{-3}/(k+1)$}
  \label{table:2}
  \subcaptionbox{$\alpha_k:=1/(k+1)$}{\begin{tabular}{c|cc}
\hline
Group & $F_{G_{j}}$ & $D_{G_{j}}$ \\
\hline
01 & 163.488124 & 0.319108 \\
02 & 158.709368 & 0.299227 \\
03 & 154.986973 & 0.332485 \\
04 & 157.928141 & 0.324662 \\
05 & 169.901684 & 0.321779 \\
06 & 160.636388 & 0.308214 \\
07 & 161.228299 & 0.314448 \\
08 & 148.436659 & 0.362904 \\
09 & 160.756450 & 0.266202 \\
10 & 160.449748 & 0.368771 \\
11 & 162.608469 & 0.340017 \\
12 & 158.361547 & 0.347438 \\
13 & 166.913645 & 0.321366 \\
14 & 157.718005 & 0.240146 \\
15 & 154.436379 & 0.348402 \\
16 & 150.187445 & 0.286736 \\
\hline
\end{tabular}}
  \subcaptionbox{$\alpha_k:=10^{-3}/(k+1)$}{\begin{tabular}{c|cc}
\hline
Group & $F_{G_{j}}$ & $D_{G_{j}}$ \\
\hline
01 & 166.798378 & 0.000058 \\
02 & 161.254440 & 0.001169 \\
03 & 157.543728 & 0.001018 \\
04 & 160.614962 & 0.000035 \\
05 & 172.259768 & 0.000000 \\
06 & 162.899171 & 0.006912 \\
07 & 163.458852 & 0.000000 \\
08 & 151.501730 & 0.000813 \\
09 & 163.477941 & 0.000000 \\
10 & 163.295545 & 0.000309 \\
11 & 164.885338 & 0.007128 \\
12 & 160.815178 & 0.025368 \\
13 & 169.319284 & 0.009938 \\
14 & 160.189071 & 0.000000 \\
15 & 156.886623 & 0.001033 \\
16 & 152.430848 & 0.000000 \\
\hline
\end{tabular}}
\end{table}


\begin{thebibliography}{10}
\providecommand{\url}[1]{{#1}}
\providecommand{\urlprefix}{URL }
\expandafter\ifx\csname urlstyle\endcsname\relax
  \providecommand{\doi}[1]{DOI~\discretionary{}{}{}#1}\else
  \providecommand{\doi}{DOI~\discretionary{}{}{}\begingroup
  \urlstyle{rm}\Url}\fi

\bibitem{b-c}
Bauschke, H.H., Combettes, P.L.: Convex Analysis and Monotone Operator Theory
  in Hilbert Spaces.
\newblock Springer (2011)

\bibitem{b-g1987}
Bertsekas, D., Gallager, R.: Data Networks.
\newblock Prentice Hall (1987)

\bibitem{bert}
Bertsekas, D.P., Nedi\'c, A., Ozdaglar, A.E.: Convex Analysis and Optimization.
\newblock Athena Scientific (2003)

\bibitem{bert1997}
Bertsekas, D.P., Tsitsiklis, J.N.: Parallel and distributed computation:
  Numerical methods.
\newblock Athena Scientific (1997)

\bibitem{cai2007}
Cai, X., Sha, D., Wong, C.K.: Time-Varying Network Optimization.
\newblock Springer (2007)

\bibitem{casa2009}
Casanova, H., Legrand, A., Robert, Y.: Parallel algorithms.
\newblock Chapman and Hall/CRC (2009)

\bibitem{censor1998}
Censor, Y., Zenios, S.: Parallel Optimization: Theory, Algorithms, and
  Applications.
\newblock Oxford Univ Pr on Demand (1998)

\bibitem{comb2007}
Combettes, P.L., Pesquet, J.C.: A {D}ouglas-{R}achford splitting approach to
  nonsmooth convex variational signal recovery.
\newblock IEEE Journal of Selected Topics in Signal Processing \textbf{1},
  564--574 (2007)

\bibitem{gold2005}
Goldsmith, A.: Wireless Communications.
\newblock Cambridge University Press, Cambridge (2005)

\bibitem{iiduka_siopt2013}
Iiduka, H.: Fixed point optimization algorithms for distributed optimization in
  networked systems.
\newblock SIAM Journal on Optimization \textbf{23}, 1--26 (2013)

\bibitem{iiduka_mp2014}
Iiduka, H.: Acceleration method for convex optimization over the fixed point
  set of a nonexpansive mapping.
\newblock Mathematical Programming \textbf{149}, 131--165 (2015)

\bibitem{iiduka_mp2015}
Iiduka, H.: Nonsmooth convex optimization over fixed point sets of
  quasi-nonexpansive mappings using parallel and incremental subgradient
  methods  (submitted a revised version to Journal)

\bibitem{iiduka_hishinuma_siopt2014}
Iiduka, H., Hishinuma, K.: Acceleration method combining broadcast and
  incremental distributed optimization algorithms.
\newblock SIAM Journal on Optimization \textbf{24}, 1840--1863 (2014)

\bibitem{ken2011}
Kennington, J., Olinick, E., Rajan, D. (eds.): Wireless Network Design.
\newblock Springer (2011)

\bibitem{lee2013}
Lee, S., Nedi\'c, A.: Distributed random projection algorithm for convex
  optimization.
\newblock IEEE Journal of Selected Topics in Signal Processing \textbf{7},
  221--229 (2013)

\bibitem{mail}
Maill\'e, P., Toka, L.: Managing a peer-to-peer data storage system in a
  selfish society.
\newblock IEEE Journal on Selected Areas in Communications \textbf{26}(7),
  1295--1301 (2008)

\bibitem{nedic2011}
Nedi\'c, A.: Random algorithms for convex minimization problems.
\newblock Mathematical Programming \textbf{129}, 225--253 (2011)

\bibitem{nedic}
Nedi\'c, A., Ozdaglar, A.: Cooperative distributed multi-agent optimization.
\newblock Convex Optimization in Signal Processing and Communications pp.
  340--386 (2010)

\bibitem{pop2012}
Pop, P.C.: Generalized Network Design Problems.
\newblock Walter de Gruyter GmbH (2012)

\bibitem{ram2012}
Ram, S.S., Nedi\'c, A., Veeravalli, V.V.: A new class of distributed
  optimization algorithms: Application to regression of distributed data.
\newblock Optimization Methods and Software \textbf{27}, 71--88 (2012)

\bibitem{sri}
Srikant, R.: The Mathematics of Internet Congestion Control.
\newblock Birkhauser (2004)

\bibitem{vaa2007}
van~der Vaart, A.W.: Asymptotic Statistics.
\newblock Cambridge Series in Statistical and Probabilistic Mathematics.
  Cambridge University Press (2007)

\bibitem{wang2015}
Wang, M., Bertsekas, D.P.: Incremental constraint projection-proximal methods
  for nonsmooth convex optimization.
\newblock SIAM Journal on Optimization  (to appear)

\bibitem{yamada2011}
Yamada, I., Yukawa, M., Yamagishi, M.: Minimizing the {M}oreau envelope of
  nonsmooth convex functions over the fixed point set of certain
  quasi-nonexpansive mappings.
\newblock In: H.H. Bauschke, R.S. Burachik, P.L. Combettes, V.~Elser, D.R.
  Luke, H.~Wolkowicz (eds.) Fixed-Point Algorithms for Inverse Problems in
  Science and Engineering, pp. 345--390. Springer (2011)

\bibitem{yin2014_2}
Yin, X.C., Huang, K., Hao, H.W., Iqbal, K., Wang, Z.B.: A novel classifier
  ensemble method with sparsity and diversity.
\newblock Neurocomputing \textbf{134}, 214--221 (2014)

\bibitem{yin2014}
Yin, X.C., Huang, K., Yang, C., Hao, H.W.: Convex ensemble learning with
  sparsity and diversity.
\newblock Information Fusion \textbf{20}, 49--58 (2014)

\end{thebibliography}
\end{document}